\documentclass[12pt]{amsart}
\usepackage{amsmath}
\usepackage{amsxtra}
\usepackage{amscd}
\usepackage{amsthm}
\usepackage{amsfonts}
\usepackage{amssymb}
\usepackage {pstricks}
\usepackage{pstricks,pst-node}
\usepackage[all]{xy}

\usepackage{enumerate}

\usepackage{mathrsfs}

\usepackage{verbatim}

\usepackage[new]{old-arrows}

\newtheorem{theorem}{Theorem}[section]
\newtheorem*{theorem*}{Theorem}
\newtheorem{lemma}{Lemma}[section]

\newtheorem*{proposition*}{Proposition}
\newtheorem{corollary}{Corollary}[section]

\theoremstyle{definition}
\newtheorem{definition}{Definition}[section]
\newtheorem{remark}{Remark}[section]
\newtheorem*{conjecture*}{Conjecture}
\newtheorem*{notation*}{Notation}

\newtheorem{example}{Example}[section]

\newtheorem{fact}{Fact}[section]

\numberwithin{equation}{section}

\def\1{1\kern-.3em1}
\newcommand{\Z}{{\mathbb Z}}

\newcommand{\C}{{\mathbb C}}
\newcommand{\K}{{\mathbb K}}

\newcommand{\U}{{\rm U}}

\newcommand{\fg}{{\mathfrak g}}

\newcommand{\fsl}{{\mathfrak{sl}}}

\newcommand{\id}{{\rm{id}}}
\newcommand{\Hom}{{\rm{Hom}}}

\newcommand{\cD}{{\mathscr D}}

\newcommand{\cC}{{\mathcal C}}

\newcommand{\cP}{{\mathcal P}}

\newcommand{\GL}{{\rm GL}}
\newcommand{\cL}{{\mathcal L}}
\newcommand{\cA}{{\mathcal A}}

\newcommand{\Sp}{{\rm Sp}}

\newcommand{\lra}{\longrightarrow}
\newcommand{\wh}{\widehat}
\newcommand{\wt}{\widetilde}

\newcommand{\beq}{\begin{eqnarray}}
\newcommand{\eeq}{\end{eqnarray}}

\newcommand{\baln}{\begin{aligned}}
\newcommand{\ealn}{\end{aligned}}

\newcommand{\bmtx}{\begin{pmatrix}}
\newcommand{\emtx}{\end{pmatrix}}

\newcommand{\brmk}{\begin{remark}}
\newcommand{\ermk}{\end{remark}}

\newcommand{\SD}{\cD}
\newcommand{\cJ}{{\mathcal J}}

\newcommand{\Pf}{{\rm{Pf}}}

\begin{document}
\title[Quasi Laurent polynomial algebras]
{Automorphisms and representations of \\ quasi Laurent polynomial algebras}
\author[He Zhang]{He Zhang}
\address{Department of Mathematical Sciences, Tsinghua University, Beijing,  China}
\email{he-zhang17@mails.tsinghua.edu.cn}
\author[Hechun Zhang]{Hechun Zhang}
\thanks{Hechun Zhang was partially supported by National Natural Science Foundation of China grants No. 12031007 and No. 11971255}
\address{Department of Mathematical Sciences, Tsinghua University, Beijing,  China}
\email{hczhang@tsinghua.edu.cn}
\author[R. B. Zhang]{Ruibin Zhang}
\thanks{R.B. Zhang was partially supported by Australian Research Council grant DP170104318.}
\address{School of Mathematics and Statistics,
The University of Sydney, Sydney, NSW 2006,  Australia}
\email{ruibin.zhang@sydney.edu.au}

\begin{abstract}
We study automorphisms and representations of quasi polynomial algebras (QPAs) and quasi Laurent polynomial algebras (QLPAs).
For any QLPA defined by an arbitrary skew symmetric integral matrix, we explicitly describe its automorphism groups at generic $q$ and at roots of unity.  Any QLPA is isomorphic to the tensor product of  copies of the QLPA of degree $2$ at different powers of $q$ and the centre, thus
the study of representations of QPAs and QLPAs largely reduces to that of $\cL_q(2)$ and $\cA_q(2)$, the
QLPA and QPA of degree $2$. We study a category of $\cA_q(2)$-modules which have finite covers by submodules with natural local finiteness properties and satisfy some condition under localisation, determining its blocks,  classifying the simple objects and providing two explicitly constructions for the simples.  One construction
produces the simple $\cA_q(2)$-modules from $\cL_q(2)$-modules via monomorphisms composed of the natural embedding of $\cA_q(2)$ in $\cL_q(2)$ and automorphisms of $\cL_q(2)$, and the other explores a class of holonomic $\SD_q$-modules for the algebra  $\SD_q$ of $q$-differential operators.
\end{abstract}
\maketitle

%\tableofcontents

\section{Introduction}

\subsection{ }
\addtocontents{toc}{\protect\setcounter{tocdepth}{1}}

Quasi polynomial algebras (QPAs) were investigated by De Concini and Procesi in \cite{DP-LNM} for the purpose of computing degrees of quantum coordinate algebras at a root of unity (see \cite{DL, JZ} for example).
It turns out that QPAs also play a crucial role in studying representations of function algebras of quantum groups \cite{Joseph, S} at a generic parameter, and provide a key ingredient in the theory of quantum cluster algebras \cite{BZ}.
In fact one unavoidable issue in the theory of quantum groups is to treat algebras obeying $q$ commutation relations, for instance, in the construction of  canonical bases \cite{ZZ05, Z}, the study of representations at a root of unity \cite{DK1, DKP1, DKP2, DP, DP-LNM, JZ}, and other aspects of the theory. Note also that 
 the fraction field of any quantum solvable algebra (with some  assumption) \cite{Panov}, in particular, that of the positive part of a quantum group of type-$A$ \cite{IM},  is isomorphic to the field of fractions of a QPA.

Corresponding to each $n\times n$ skew-symmetric integral matrix $H$, there is a QPA denoted by $\cA_q[H]$, which is a twisted polynomial algebra.  
By inverting the generators of $\cA_q[H]$, one obtains the quasi Laurent polynomial algebra (QLPA) $\cL_q[H]$. 
The QPA and QLPA corresponding to $H=\bmtx 0 & 1\\ -1 & 0\emtx$,  denoted by by $\cA_q(2)$ and $\cL_q(2)$ respectively, attracted much research \cite{DP, JZ,KPS, PLCN, Tang} because of their important roles in the theory of quantum groups. The algebra $\cL_q(2)$ is also known as the quantum torus of rank $2$, which appears in many contexts.

We study automorphisms and representations of QLPAs and QPAs in this paper.

\subsection{ }
\addtocontents{toc}{\protect\setcounter{tocdepth}{1}}

In the first part of this paper, we determine the automorphism groups of $\cL_q[H]$ for any $H$ both for $q$ being generic and a root of unity. The main results are given in Theorem \ref{thm:auto-main}  and Theorem \ref{thm:auto-main-roots}, which solve the problem of automorphism groups of QLPAs in full generality.

We remark that the automorphism group of $\cL_q(2)$ was determined in \cite{KPS}, and also in \cite{PLCN} where the automorphism groups were described for a class of quantum algebras including the quantum torus of rank $2$.

To provide some context for our study of automorphisms of QLPAs, we mention that the automorphism group of the polynomial algebra in two variables was much studied, and an explicit description of it has long been known. There were different proofs of this classical result, see, e.g., Nagata \cite{N}, McKay and Wang \cite{MW},  where the authors tried to better understand the two variable case with the aim of extending the result  to more variables. One naturally expects that an understanding of the automorphism group will help to address the celebrated Jacobian conjecture.

\subsection{ }
\addtocontents{toc}{\protect\setcounter{tocdepth}{1}}

Our main motivation of studying the automorphism groups of QLPAs in this paper comes from representation theory.  If one is able to construct,  for example,  an algebra homomorphism  $A\stackrel{\varphi}{\longrightarrow}  \cL_q[H]$ for an associative algebra $A$, then obviously any $\cL_q[H]$-module naturally becomes an $A$-module via $\varphi$. There is enormous flexibility in this furnished by automorphisms $\tau$ of $\cL_q[H]$ through the composed algebra homomorphisms 
$A\stackrel{\varphi}{\longrightarrow}  \cL_q[H]\stackrel{\tau}{\longrightarrow} \cL_q[H]$.  
This simple fact is explored to great effect in Section \ref{sect:restrict} to construct simple $\cA_q[H]$-modules by investigating $\cL_q[H]$-modules via monomorphisms
$\cA_q[H]\stackrel{\iota}{\longhookrightarrow} \cL_q[H]\stackrel{\tau}{\longrightarrow} \cL_q[H]$ with $\iota$ being the natural embedding.  

We point out that the minimal cyclic modules for the quantum matrix algebra when $q$ is a root of unity were classified following this strategy \cite{JZ-m}. 
The quantum Fock space representations  \cite{H}  of quantum groups of classical types
were obtained via some algebra homomorphisms from the quantum groups to quantum Weyl algebras and quantum Clifford algebras. Recently, a cluster algebra flavour was added into this strategy to study the positive representations of quantum groups using the Laurent Phenomenon of quantum cluster algebras \cite{Ip}. 

\subsection{ }
\addtocontents{toc}{\protect\setcounter{tocdepth}{1}}

We study representations of QPAs and QLPAs in the second part of this paper.

There is extensive literature on the theory of $q$-difference modules (see, e.g.,  \cite{dV, RSZ} and references therein), which is the $q$-difference analogue of $\SD$-module theory in dimension $1$. There is the algebra $\SD_q$ of $q$-difference operators in dimension $1$, which is an adic completion of $\cL_q(2)$. Every $q$-difference module is a $\SD_q$-module that is holonomic in the sense that its
Gelfand–Kirillov dimension is equal to $1$. 
In \cite{BEG}, a BGG-type highest weight category of infinite dimensional modules of double Hecke algebras was studied using a class of holonomic  $\SD_q$-modules.  Recently Kontsevich and Soibelman established  a Riemann-Hilbert correspondence for holonomic $\SD_q$-modules and investigated $\SD_q$-modules with anti Hardamard filtrations.

We prove in Corollary \ref{cor:reduct} a structure theorem for QLPAs,  which states that  $\cL_q[H]$ is isomorphic to the tensor product of $\cL_{q^{m_i}}(2)$'s and the centre, which is a multi-variable Laurent polynomial algebra.  Thus the canonical  quasi polynomial subalgebra of the QLPA in the factorised form is the tensor product of $\cA_{q^{m_i}}(2)$'s and a multi-variable polynomial algebra which is the centre of the subalgebra. Therefore, the study of representations of QLPAs and QPAs, to a large extend, boils down to understanding representations of $\cL_q(2)$ and $\cA_q(2)$.

If $q$ is a root of $1$, both $\cL_q(2)$ and $\cA_q(2)$ are finite dimensional over their centres, and their representations are fairly easy to understand. However, at generic $q$, all simple modules are infinite dimensional except for some $1$-dimensional modules for $\cA_q(2)$, and it is quite hard to construct such infinite dimensional modules, especially for $\cA_q(2)$. Thus we focus on the representation theory of $\cA_q(2)$ at generic $q$.

%$\bullet$ 
We introduce a notion of admissible $\cA_q(2)$-modules, which have finite covers by submodules with natural local finiteness properties and satisfy some condition under localisation.
We determine in Lemma \ref{lem:cat-adm} the blocks of the category $\cC$ of admissible $\cA_q(2)$-modules, which are parametrised by the set $\cP$ of pairs of co-prime non-negative integers. We show in Corollary \ref{cor:simple-C} that the simple objects in the block corresponding to any $(a, b)\in\cP/\{(1,0), (0, 1)\}$ are parametrised by points of an elliptic curve ${\rm E}_q:=\C^*/q^\Z$ with  $q^\Z=\{q^m\mid m\in\Z\}$.

%$\bullet$  
We develop two different constructions for the simple objects of $\cC$. One of the constructions is given in Theorem \ref{thm:restr} and Corollary \ref{cor:full-set}, which produces the simple admissible $\cA_q(2)$-modules from $\cL_q(2)$-modules via monomorphisms composed of the natural embedding of $\cA_q(2)$ in $\cL_q(2)$ and automorphisms of $\cL_q(2)$, making essential use of results obtained on the automorphism group.  The other construction is given in Theorem \ref{thm:construct-d-diff} in the context of holonomic $\SD_q$-modules.

\subsection{ }
\addtocontents{toc}{\protect\setcounter{tocdepth}{1}}

The arrangement of the paper is as follows. In Section 2, we give the basic definitions of QPAs and QLPAs, and then prove the structure theorem for QLPAs, which states that $\cL_q[H]$ is isomorphic to the tensor product of $\cL_{q^{m_i}}(2)$'s with the centre which is a Laurent polynomial algebra. In section 3, we describe explicitly the automorphism groups of $\cL_q[H]$ both for generic $q$ and roots of unity. In Section 4 we prove some structural properties of $\cA_q(2)$, which will play important roles in studying representations of this QPA. In Section 5, we define the category $\cC$ of admissible $\cA_q(2)$-modules, determine its blocks, and classify the simple objects. In Section 6, we develop two different constructions for the simple objects of $\cC$, respectively by using $\cL_q(2)$-modules twisted by automorphism of the algebra, and by using holonomic $\SD_q$-modules.  Some basic material on $\SD_q$-modules is also presented in this section.

%\subsection{ }
\addtocontents{toc}{\protect\setcounter{tocdepth}{2}}

%%%
%%%
\part{Automorphism groups}
%%%
%%%

In this part of the paper, we determine the automorphism groups of any quasi Laurent polynomial algebra for generic $q$ and for $q$ being any root of unity.

\section{Quasi polynomial and Laurent polynomial  algebras}
%\addtocontents{toc}{\protect\setcounter{tocdepth}{2}}
%\subsection{Quasi polynomial  algebras and quantum tori}
Given any $n\times n$ skew-symmetric matrix $H=(h_{ij})$ over ${\Z}$, one
constructs an associative algebra $\cA_q[H]$ with identity, called  a quasi polynomial algebra (QPA), which is the generated by elements $x_1,x_2,\cdots,x_n$ with
the following defining relations, 
\begin{equation}\label{eq:Aq-relations}
x_ix_j=q^{h_{ij}}x_jx_i, \quad \text{ for }i,j=1, 2, \cdots, n,
\end{equation}
where $q\in\C^*$ is a fixed parameter. %We call the matrix $H$ the defining matrix of $\cA_q[H]$.
The algebra $\cA_q[H]$ can be viewed as an iterated twisted polynomial algebra with
respect to any ordering of the indeterminates $x_i$. Given
any $a=(a_1,a_2,\cdots,a_n)^T\in\Z_+^n$,  we write
$x^a=x_1^{a_1}x_2^{a_2}\cdots x_n^{a_n}$. Such ordered monomials form a basis of $\cA_q[H]$.

The algebra $\cA_q[H]$ admits a torus action. Corresponding to each element $\underline{c}=(c_1, c_2, \cdots, c_n)^T\in {\C^*}^n$, there exists an automorphism,  denoted by $\tau_{\underline{c}}$, given by
\begin{eqnarray*}
\tau_{\underline{c}}:  \cA_q[H]\longrightarrow \cA_q[H], \quad
x_i\mapsto c_ix_i, \quad  i=1, 2,\cdots, n.
\end{eqnarray*}

The quasi Laurent polynomial algebra (QLPA) corresponding to $H$, denoted by $\cL_q[H]$,  is the associative algebra generated by $x_1, x_2, \cdots, x_n$ and their respective inverses $x_1^{- 1}, x_2^{- 1}, \cdots, x_n^{-1}$ with the defining relations
\eqref{eq:Aq-relations}. The Laurent monomials
$x^{\underline{a}}=x_1^{a_1}x_2^{a_2}\cdots x_n^{a_n}$ with $\underline{a}$ $=$ $(a_1,a_2,\cdots,a_n)^T\in\Z^n$ form a basis of $\cL_q[H]$. The torus action on $\cA_q[H]$ can be extended to $\cL_q[H]$ in the obvious way.

Let $S=\begin{pmatrix}0 & 1\\ -1 & 0\end{pmatrix}$.
A particularly important example is the QLPA corresponding to
$H=S$, denote by $\cL_q(2)$, which is the quantum torus of rank $2$. The  corresponding QPA will be denoted by $\cA_q(2)$.

Denote by $M_{m\times n}(\Z)$ the set of $m\times n$ matrices over $\Z$, and write $M_n(\Z)=M_{n\times n}(\Z)$. The general linear group $\GL_n(\Z)$ over $\Z$ is a subset of $M_n(\Z)$ consisting of non-singular elements. Note that a matrix $A\in M_n(\Z)$ is non-singular  if and only if $\det A=\pm 1$.   Observe the following fact.

\begin{theorem}\label{thm:iso}
$\cL_q[H]\cong \cL_q[H^\prime]$ if there exists $P\in \GL_n(\Z)$ such that $P^THP=H^\prime$, 
where $P^T$ is the transpose of $P$.
\end{theorem}

\begin{proof}
Denote by $x_1, x_2, \cdots, x_n$ the generators of $\cL_q[H]$, which satisfy the defining relations \eqref{eq:Aq-relations} with respect to  the skew-symmetric integral matrix $H$. Write $P=(\alpha_1, \alpha_2, \cdots, \alpha_n)$ in column block form. Let $y_i=x^{\alpha_i}$ (for $i=1, 2, \cdots, n$),
which are Laurent monomials of the generators $x_1, x_2, \cdots, x_n$ since $P\in GL_n(\Z)$.
Then $y_1, y_2, \cdots, y_n$ satisfy the defining relation of $\cL_q[H^\prime]$ with $H^\prime = P^THP$.   Hence we have the following algebra homomorphism
\begin{eqnarray}
\tau_P: \cL_q[H^\prime]\longrightarrow \cL_q[H], \quad
x^\prime_i\mapsto y_i, \quad \forall \quad i,
\end{eqnarray}
where $x^\prime_i$ are the generators of  $\cL_q[H^\prime]$ satisfying the defining relations \eqref{eq:Aq-relations} with respect to $H^\prime$. The fact $(P^{-1})^TH^\prime P^{-1}=H$ implies that $\tau_P$ has an inverse map.
\end{proof}

A skew-symmetric matrix $H\in M_n(\Z)$ can be brought into a block diagonal form by a matrix
$W\in \GL_n(\Z)$. Specifically, there is  a sequence of positive integers $m_1\le m_2\le \dots\le m_N$, where
$2N$ is equl to the rank $r(H)$ of $H$, such that
\begin{equation}\label{eq:block-diagonal}
W H W^T= diag(m_1 S,  \dots,  m_N S, 0, \dots, 0).
\end{equation}
\begin{definition} \label{def:canonical}
Denote the right-hand-side of equation \eqref{eq:block-diagonal} by $H_{can}$, and call it the canonical form of $H$.
\end{definition}

As a corollary of Theorem \ref{thm:iso}, we have the following result.
\begin{corollary}\label{cor:reduct}
Retain notation above. As algebras
\[
\cL_q[H]\cong \cL_{q^{m_1}}(2)\otimes \cL_{q^{m_2}}(2)\otimes\cdots\otimes \cL_{q^{m_N}}(2)\otimes Z,
\]
where $Z$ is the centre of $\cL_q[H]$ which is a Laurent polynomial algebra in $n-r(H)$ variables with a basis $\{x^\alpha\mid \alpha\in \Z^n, H\alpha=0\}$.
\end{corollary}

The following result is easy to see. It will play an important role later.
\begin{theorem} \label{thm:units}
Every unit in $\cL_q[H]$ is a non-zero scalar multiple of some Laurent monomial in the generators.
\end{theorem}

\begin{proof} The  set of Laurent monomials $\mathcal B=\{x^{\alpha}\mid \alpha\in \mathbb Z^n\}$ is a basis of the algebra $\mathcal L_q[H]$. It is ordered with the lexicographic order such that $x^\alpha>x^\beta$ if the first non-zero entry of $\alpha-\beta$ is positive. 

Clearly, any non-zero scalar multiple of a Laurent monomial in the generators is an unit. We will show that any unit must be of this form. If, on the contrary, an element $u=\sum_{i=1}^s a_i x^{\alpha_i}$ with $s\ge 2$ and all  $a_i\ne 0$ is a unit, in which $x^{\alpha_s}$ is maximal and $x^{\alpha_1}$ is minimal, then $u^{-1}=\sum_{j=1}^t b_j x^{\beta_j}$ in which $x^{\beta_t}$ is maximal and $x^{\beta_1}$ is minimal. Now the maximal term in $u u^{-1}=1$ is a non-zero multiple of $x^{\alpha_s+\beta_t}$, and the minimal term is a non-zero multiple of $x^{\alpha_1+\beta_1}$, which is absurd.
\end{proof}

\section{Automorphism groups of quasi Laurent polynomial  algebras}

Recall that a skew symmetric matrix $H\in M_n(\Z)$ can be brought into a block diagonal form,  the canonical form $H_{can}$,  by some $W\in \GL_n(\Z) $ as described in \eqref{eq:block-diagonal}.  Note that $H_{can}$ is uniquely determined by a sequence of positive integers $m_1\le m_2\le \cdots\le m_N$, where the rank $r$ of $H$ is equal to $2N$. We write
$\underline{m}=(m_1, m_2, \cdots, m_N)$, and say that the matrix $H$ is of type $(n, r, \underline{m})$.
Denote $\Lambda(\underline{m})=diag(m_1S, m_2S, \cdots, m_NS)$, which is the non-degenerate part of $H_{can}$. Then $H_{can}=\begin{pmatrix}\Lambda(\underline{m}) & 0 \\ 0 & 0\end{pmatrix}$. By theorem \ref{thm:iso}, we  have
\[
\cL_q[H]\cong \cL_q[H_{can}]=\cL_q[\Lambda(\underline{m})]\otimes_{\C} Z,
\]
where $Z$ is a Laurent polynomial algebra in $n-r$ variables, which is the centre of $\cL_q[H]$.   Denote by $y_1^{\pm 1}, y_2^{\pm 1}, \cdots, y_r^{\pm 1}$ the generators of $\cL_q[\Lambda(\underline{m})]$, and by $y_{r+1}^{\pm 1}, y_{r+2}^{\pm 1}, \dots, y_n^{\pm 1}$ the generators of $Z$. Write $[k, \ell]=\{k, k+1, \dots, \ell\}$ for any integers $\ell>k\ge 0$.  Then
\beq\label{eq:y-relats}
y_a y_b =  q^{(H_{can})_{a b}} y_b y_a, \quad \forall a, b\in[1, n],
\eeq
where $(H_{can})_{a b}$ is the $(a, b)$-entry of $H_{can}$.
This can be written more explicitly as
\[
\baln
&y_i y_j = q^{\lambda_{i j}} y_j y_i, \quad \forall i, j \in[1, r], \\
&y_a y_s = y_s y_a, \quad \forall \ s \in [r+1,  n], \ a\in[1, n].
\ealn
\]
where $\lambda_{i j} = \Lambda(\underline{m})_{i j}$.

To describe the automorphisms of $\cL_q[H]$, it suffices to do that for $\cL_q[H_{can}]$.

\subsection{The automorphism group at generic $q$}
Assume that $q$ is generic.
The general linear group $\GL_n(\C)$ contains a maximal parabolic subgroup
\[
P_r(\C)=\left\{\left.\begin{pmatrix}A & 0 \\ C & D\end{pmatrix}\in \GL_n(\C)\right|A\in \GL_r(\C), D\in \GL_{n-r}(\C), C\in M_{r\times (n-r)}(\C)\right\}.
\]
Denote by $\Sp(\underline{m},  \C)$ the symplectic group consisting of matrices $A\in \GL_r(\C)$ such that
$
A^T\Lambda(\underline{m})A=\Lambda(\underline{m}).
$
Taking the Pfaffian of both sides of the relation, we obtain $\Pf(A^T\Lambda(\underline{m})A)=\det(A) \Pf(\Lambda(\underline{m}))=\Pf(\Lambda(\underline{m}))$. Thus $\det(A)=1$, and hence $\Sp(\underline{m},  \C)\subset {\rm SL}_r(\C)$. The following subset  of $P_r(\C)$,
\[
Q(\underline{m},  \C)=\left\{\left.\begin{pmatrix}A & 0 \\  C & D\end{pmatrix}\in P_r(\C)\right|A\in \Sp(\underline{m},  \C)\right\},
\]
forms a subgroup. Indeed for any  elements $M=\begin{pmatrix}A & 0 \\  C & D\end{pmatrix}$ and $M'=\begin{pmatrix}A' & 0 \\  C' & D'\end{pmatrix}$ in $Q(\underline{m},  \C)$, we have $M M' = \begin{pmatrix}A A' & 0 \\  CA' + D C' & D D'\end{pmatrix}$. Since $A A'\in \Sp(\underline{m},  \C)$, we have $M M'\in Q(\underline{m},  \C)$.
Let
\[
\Sp(\underline{m},  \Z)= \Sp(\underline{m},  \C)\cap {\rm SL}_r(\Z),  \quad
Q(\underline{m},  \Z)=Q(\underline{m},  \C)\cap \GL_n(\Z).
\]
%\[
%Q(\underline{m},  \Z)=\left\{\left.\begin{pmatrix}A & 0 \\ C & D\end{pmatrix}\in P_r(\C)
%\bigcap \GL_n(\Z)\right|A\in \Sp(\underline{m},  \Z)\right\}.
%\]

\begin{theorem} \label{thm:auto-main} Retain notation above.
Assume that a skew symmetric matrix $H\in M_n(\Z)$ of rank $r$ is in canonical form with non-degenerate part  $\Lambda(\underline{m})$.
Then the automorphism group of $\cL_q[H]$ is given by
\[
Aut(\cL_q[H])=\left\{\tau_{\underline{c}}\tau_M \mid \underline{c}\in (\C^*)^n,  M\in Q(\underline{m},  \Z) \right\},
\]
where $\tau_{\underline{c}}$ and $\tau_M$ for any $\underline{c}=(c_1, c_2, \cdots, c_n)^T\in ({\C^*})^n$, and $M=(\alpha_1, \alpha_2,\cdots, \alpha_n)\in Q(\underline{m},  \Z)$ are respectively defined by
\beq\label{eq:maps}
\tau_{\underline{c}}(y_i)=c_i y_i, \quad \tau_M(y_i)=y^{\alpha_i}, \quad   \forall i=1, 2, \dots, n.
\eeq
\end{theorem}
\begin{proof}
For any $\tau\in Aut(\cL_q[H])$, it follows $y_i y_i^{-1}=1$  that $\tau(y_i)$ is an invertible element in $\cL_q[H]$ for all $i$. By Theorem \ref{thm:units}, it must be of the form
\beq\label{eq:monomial}
\tau(y_i)=\tilde{c}_i y^{\alpha_i} \quad \text{for some }  \tilde{c}_i\in \C^*,  \alpha_i\in \Z^n.
\eeq
Now $\underline{\tilde c}=(\tilde{c}_1, \dots, \tilde{c}_n)^T\in (\C^*)^n$ and  $M=(\alpha_1, \dots, \alpha_n)\in M_n(\Z)$.

Write $Y_i=\tau(y_i)$ for all $i$, then it again follows Theorem \ref{thm:units} that $\tau^{-1}$ is given by
\[
\tau^{-1}(Y_i) = c_i' Y^{\beta_i}, \quad \forall i,
\]
for some $\underline{c'}=(c_1', \dots, c_n')\in (\C^*)^n$ and $M'=(\beta_1, \dots, \beta_n)\in M_m(\Z)$.  As $\tau^{-1}(Y_i)=y_i$ for all $i$, we must have $MM'=I_n$ (the identity matrix), and hence $M\in\GL_n(\Z)$.

As $\tau(y_i)$ satisfy the relations \eqref{eq:y-relats}, we must have $q_{i j} ^{(M^T H M)_{i j}} = q^{H_{i j}}$ for all $i, j$. For  generic $q$, this holds if and only if
$
M^T H M = H,
$
that is,
\beq\label{eq:deg-form}
M^T \begin{pmatrix}\Lambda(\underline{m}) & 0 \\ 0 & 0\end{pmatrix}M
= \begin{pmatrix}\Lambda(\underline{m}) & 0 \\ 0 & 0\end{pmatrix}.
\eeq
Write $M$ in block form $M= \begin{pmatrix}A & B \\ C & D\end{pmatrix}$ with $A\in M_r(\Z)$,  $D\in M_{n-r}(\Z)$ and $B, C^T\in M_{r\times(n-r)}(\Z)$. Then equation \eqref{eq:deg-form} becomes
\[
A^T \Lambda(\underline{m}) A = \Lambda(\underline{m}), \quad B^T \Lambda(\underline{m}) A=0, \quad B^T \Lambda(\underline{m}) B=0.
\]
Thus $A\in \Sp(\underline{m},  \Z)$ and $B=0$, and hence $M\in Q(\underline{m},  \Z)$.
\end{proof}

\begin{remark} \label{rmk:gp-ind}
For any $M, M'\in Q(\underline{m},  \Z)$, we have $\tau_M\tau_{M'} =\tau_{\underline{c}_{M M'}} \tau_{M M'}$ for some $\underline{c}_{M M'}\in (\C^*)^n$. The subgroup
$
T:=\left\{\tau_{\underline{c}}\mid \underline{c}\in (\C^*)^n\right\}
$
is normal, and
$
Aut(\cL_q[H])/ T\cong Q(\underline{m},  \Z).
$
\end{remark}

\begin{remark} A similar argument shows that any anti-automorphism of $\cL_q[H]$ can be written as $\tau_{\underline{c}}\eta_M$, $\underline{c}\in (\mathbb C^*)^n$, $M=\begin{pmatrix}A&0\\C&D\end{pmatrix}\in \GL_n(\mathbb Z)$ with $A\in M_r(\mathbb Z)$, $C\in M_{(n-r)\times r}(\mathbb Z)$ and  $D\in M_{(n-r)}(\mathbb Z)$ satisfying
\[
A^T\Lambda(\underline{m}, \underline{r})A=-\Lambda(\underline{m}, \underline{r}).
\]
The anti-automorphism $\eta_M$ is defined in a similar manner as that of $\tau_M$. Using the Pfaffian, one can show that $det(A)=(-1)^{\frac{r}{2}}$\end{remark}

We can easily extract the automorphism groups of $\cL_q[\Lambda(\underline{m})]$ and of the Laurent polynomial algebra $Z=\C[y_{r+1}^{\pm 1}, y_{r+2}^{\pm 1}, \dots, y_n^{\pm 1}]$ from
Theorem \ref{thm:auto-main}.

Restricting the first map in \eqref{eq:maps} to $\cL_q[\Lambda(\underline{m})]$, we obtain
$\tau_{\underline{c}}\in Aut(\cL_q[\Lambda(\underline{m})])$ (slightly abusing notation) for any $\underline{c}\in ({\C^*})^r$. Also,
embedding $\Sp(\underline{m},  \Z)$ in $Q(\underline{m},  \Z)$ by $A\mapsto \begin{pmatrix}A &0\\0 & 1\end{pmatrix}$, we obtain $\tau_A\in Aut(\cL_q[\Lambda(\underline{m})])$ (again slightly abusing notation again) by restricting to $\cL_q[\Lambda(\underline{m})]$ the second map in \eqref{eq:maps} associated with this matrix.

As an immediate consequence of Theorem \ref{thm:auto-main}, we have the following result.

\begin{corollary} Retain notation above.
Any automorphism $\tau$ of  $\cL_q[\Lambda(\underline{m})]$ can be written uniquely as $\tau=\tau_{\underline{c}}\tau_A$ with $ \underline{c}\in (\C^*)^r$ and $A\in \Sp(\underline{m},  \Z)$.
\end{corollary}

\begin{remark}
Since $\Sp(2, \C)\cap M_2(\Z)={\rm SL}_2(\Z)$, we recover the automorphism group of $\cL_q(2)$ obtained by  Kirkman, Procesi, and Small  in \cite{KPS}.  There was a minor slip in \cite{KPS} stating that the automorphisms of $\cL_q(2)$ induced by ${\rm SL}_2(\Z)$ form a subgroup, which in fact is not the case as can be seen from Remark \ref{rmk:gp-ind}.
\end{remark}

By restricting the maps \eqref{eq:maps} to the centre $Z$ of $\cL_q[H]$, and recalling that $Z$ is the algebra of  Laurent polynomials in $m=n-r$ variables, we easily recovered the following well-known result (see, e.g., \cite{PLCN}).
\begin{corollary}\label{cor:auto-Laurent}
$Aut(\C[x_1^{\pm 1}, x_2^{\pm 1}, \cdots, x_m^{\pm 1}])\cong{\C^*}^n\rtimes \GL_m(\Z)$.
\end{corollary}

\subsection{The automorphism group at roots of unity}
Retain notation of the last section, in particular,  that for the generators of $\cL_q[H]=\cL_q[\Lambda(\underline{m})]\otimes Z$. 

We take $q$ to be a primitive $\ell$-th root of $1$, which is $\Lambda(\underline{m})$-coprime in the sense that
$(m_i, \ell)=1$ for all $i$, where $m_i$ are the components of $\underline{m}=(m_1, \dots, m_t)$.
Note that $y_i^{\pm\ell}$ are central for all $i=1, 2, \dots, r$, and the centre of $\cL_q[H]$ in this case is equal to
\[
\wh{Z}_q=\C[y_1^{\pm\ell}, \dots, y_r^{\pm\ell}]\otimes Z=\C[y_1^{\pm\ell}, \dots, y_r^{\pm\ell}, y_{r+1}^{\pm 1}, \dots y_{n}^{\pm 1}].
\]

Let us define the following subset of $\GL_n(\Z)$ for $\ell$ being $\Lambda(\underline{m})$-coprime.
\[
\baln
Q_\ell(\underline{m},  \Z):= & \text{ set of matrices $M=\begin{pmatrix}A & B \\  C & D\end{pmatrix}\in \GL_n(\Z)$ with }\\
& \text{$A\in M_r(\Z)$, $B, C^T\in M_{r\times (n-r)}(\Z)$,  $D\in M_{n-r}(\Z)$  such that} \\
&A^T \Lambda(\underline{m}) A = \Lambda(\underline{m}) \  { \rm mod} \ \ell, \quad
B=0 \  { \rm mod} \ \ell.
\ealn
\]

\begin{lemma}\label{lem:Ql-gp}
Assume that $\ell$ is $\Lambda(\underline{m})$-coprime. Then
$Q_\ell(\underline{m},  \Z)$ is a subgroup of $\GL_n(\Z)$. Any element $M=\begin{pmatrix}A & B \\  C & D\end{pmatrix}$ (notation as above) of  $Q_\ell(\underline{m},  \Z)$ satisfy $A\in\GL_r(\Z)$ and $\wt{D}:=D-CA^{-1}B\in\GL_{n-r}(\Z)$. Furthermore,
$\det(A)=1$ if $\ell>2$.
\end{lemma}
\begin{proof}
Denote by $\Pf(\Omega)$ the Pfaffian of any skew symmetric matrix $\Omega$. Then
\[
\det(A) \Pf( \Lambda(\underline{m}))= \Pf( \Lambda(\underline{m})) \  { \rm mod} \ \ell.
\]
Note that $\Pf( \Lambda(\underline{m}))$ is equal to the product of all $m_i$, thus is an integer co-prime to $\ell$. Hence  $\det(A) = 1 \  { \rm mod} \ \ell$, thus $A$ is non-singular.  We now show that $A\in\GL_r(\Z)$.

The inverse matrix
$
M^{-1} =\bmtx
A^{-1}(A + B\wt{D}^{-1} C) A^{-1} &  - A^{-1}B\wt{D}^{-1}\\
-\wt{D}^{-1}CA^{-1} & \wt{D}^{-1}
\emtx
$
of $M$ is an integral matrix, since $M\in\GL_n(\Z)$.
This in particular implies that $\wt{D}\in\GL_{n-r}(\Z)$. Hence $\det(\wt{D})=\pm 1$. Thus $\det(A)=\pm 1$
since $\det(M) = \det(A) \det(\wt{D})=\pm 1$, and hence $A\in\GL_r(\Z)$.
In particular,  $\det(A)=1$ if $\ell>2$.

It is now easy to prove that $Q_\ell(\underline{m},  \Z)$ is a subgroup of $\GL_n(\Z)$. Consider $M^{-1}$. Since $A^{-1}$ and $\wt{D}^{-1}$ are both integral matrices,  $A^{-1}B\wt{D}^{-1}=0 \ {\rm mod}\ \ell$. Also note that
$A^{-T} \Lambda(\underline{m})  A^{-1}$ $=$ $\Lambda(\underline{m})  \ {\rm mod}\ \ell$.
Hence
\[
\baln
&(A^{-1}(A + B\wt{D}^{-1} C) A^{-1})^T \Lambda(\underline{m})  (A^{-1}(A + B\wt{D}^{-1} C) A^{-1}) \\
&=  A^{-T} \Lambda(\underline{m}) A ^{-1} \ {\rm mod}\ \ell =  \Lambda(\underline{m})  \ {\rm mod}\ \ell.
\ealn
\]
Hence $M^{-1}\in Q_\ell(\underline{m},  \Z)$.

Let $M'=
\begin{pmatrix}A' & B' \\  C' & D'\end{pmatrix}$ be another element of $Q_\ell(\underline{m},  \Z)$. Then
\[
M M'= \begin{pmatrix}A A' + B C' & A B' + B D' \\  C A' + D C'& C B' + D D'\end{pmatrix}\in\GL_n(\Z).
\]
Clearly $A B' + B D'=0 \ {\rm mod}\ \ell$, and we also have
\[
\baln
&(A A' + B C')^T \Lambda(\underline{m})   (A A' + B C') \\
&= {A'}^T A^T \Lambda(\underline{m}) A A' \ {\rm mod}\ \ell \\
&= {A'}^T  \Lambda(\underline{m}) A' \ {\rm mod}\ \ell =  \Lambda(\underline{m})  \ {\rm mod}\ \ell.
\ealn
\]
Hence $M M'\in Q_\ell(\underline{m},  \Z)$.

This completes the proof of the lemma.
\end{proof}

We have the following result on the automorphism group of $\cL_q[H]$.
\begin{theorem} \label{thm:auto-main-roots}
Retain notation above.  Let $H\in M_n(\Z)$ be a skew symmetric matrix  of rank $r$ in canonical form with non-degenerate part $\Lambda(\underline{m})$.  Assume that $q$ is a primitive $\ell$-th root of $1$ for a $\Lambda(\underline{m})$-coprime $\ell\ge 2$. Then the automorphism group of $\cL_q[H]$ is given by
\[
Aut(\cL_q[H])=\left\{\tau_{\underline{c}}\tau_M \mid \underline{c}\in (\C^*)^n,  M\in Q_\ell(\underline{m},   \Z) \right\}.
\]
\end{theorem}
\begin{proof}
The proof is essentially identical to that of Theorem \ref{thm:auto-main}, with the only complication that the condition
\eqref{eq:deg-form} needs to be modified. Thus we will only deal with this in detail.

A matrix $M\in\GL_n(\Z)$ leads to an automorphism $\tau_M$ of $\cL_q[H]$ if and only if $q_{i j} ^{(M^T H M)_{i j}} = q^{H_{i j}}$ for all $i, j$ in view of \eqref{eq:y-relats}. This is equivalent to
$
M^T H M = H   \  { \rm mod} \ \ell,
$
that is,
\beq\label{eq:deg-form-roots}
M^T \begin{pmatrix}\Lambda(\underline{m})& 0 \\ 0 & 0\end{pmatrix}M
= \begin{pmatrix}\Lambda(\underline{m})& 0 \\ 0 & 0\end{pmatrix}  \  { \rm mod} \ \ell.
\eeq
We again write $M= \begin{pmatrix}A & B \\ C & D\end{pmatrix}$ in block form, then equation \eqref{eq:deg-form-roots} becomes
\beq
&A^T \Lambda(\underline{m}) A = \Lambda(\underline{m}) \  { \rm mod} \ \ell, \label{eq:mod-1}\\
&B^T \Lambda(\underline{m}) A=0 \  { \rm mod} \ \ell, \label{eq:mod-2} \\
&B^T \Lambda(\underline{m}) B=0 \  { \rm mod} \ \ell. \label{eq:mod-3}
\eeq

We have already seen in the proof of Lemma \ref{lem:Ql-gp} that equation \eqref{eq:mod-1} implies
$A\in\GL_r(\Z)$.  Thus \eqref{eq:mod-2} leads to $B^T \Lambda(\underline{m}) =0 \  { \rm mod} \ \ell$. Since $\det\Lambda(\underline{m})$ is co-prime to $\ell$, we conclude that
$
B =\ell \wt{B}
$
for some integral $r\times(n-r)$-matrix $\wt{B}$, and this clearly implies \eqref{eq:mod-3}.  Hence $M\in Q_\ell(\underline{m},   \Z)$.

This completes the proof of the theorem.
\end{proof}

\begin{remark} A similar argument shows that any anti-automorphism can be written as $\tau_{\underline{c}}\eta_M$, $\underline{c}\in (\mathbb C^*)^n$, $M=\begin{pmatrix}A&B\\C&D\end{pmatrix}\in \GL_n(\mathbb Z)$ in which $B^T, C\in M_{(n-r)\times r}(\mathbb Z), B=0 \  { \rm mod} \ \ell$ and $A\in M_r(\mathbb Z)$ satisfying
\[A^T\Lambda(\underline{m}, \underline{r})A=-\Lambda(\underline{m}, \underline{r})\  { \rm mod} \ \ell.\]
The anti-automorphism $\eta_M$ is defined in a similar manner as $\tau_M$. Again, by using the Pfaffian, one can show that $det(A)=(-1)^{\frac{r}{2}}$.
\end{remark}

%%%
%%%
\part{Representation theory}
%%%
We have shown in Corollary \ref{cor:reduct} that any QLPA is the tensor product of $\cL_{q^{m_i}}(2)$ factors and the centre, a multi-variable Laurent polynomial algebra.  The canonical quasi polynomial subalgebra in the factorised form of $ \cL_q[H]$ has a similar structure.
Therefore, it is the representation theory of the $\cA_{q^{m_i}}(2)$ (resp. $\cL_ {q^{m_i}}(2)$) factors which we need to develop in order to study representations of $\cA_q[H]$ (resp. $\cL_q[H]$)).

We study the representation theory of $\cA_q(2)$ and $\cL_q(2)$ in the remainder of the paper. 
When $q$ is a root of unity, this is an easy task to complete by adapting techniques and results from \cite{JZ-m} to the present setting. 

Hence we assume that $q$ is generic hereafter.

%%%
%%
\section{Structure of $\cA_q(2)$}

We examine some structural properties of the algebras $\cA_q[H]$ and $\cL_q[H]$ with the defining matrix $H$ being a $2\times 2$ skew symmetric integral matrix. Such an $H$ is either $0$ or $m S$ for a non-zero integer $m$, where $S=\begin{pmatrix}0 & 1\\ -1 & 0\end{pmatrix}$. We may take $m=1$ upon replacing $q^m$ by $q$.  

\subsection{A digression to the polynomial algebra}
If $H=0$, the algebra $\cA_q[H]$ is a polynomial algebra  $\C[x, y]$ in two variables.
Obviously, one has the following three types of automorphisms of $\C[x, y]$.

\noindent $(a).$
The flip
$
\sigma: \C[x, y]\longrightarrow \C[x, y], \quad x\mapsto y, \  y\mapsto x.
$

\noindent $(b).$
The $x$ dilation $
\tau_c: \C[x, y]\longrightarrow \C[x, y]$,  $x\mapsto cx, \  y\mapsto y, 
$
for any $c\in \C^*$.

\noindent $(c).$
The Nagata map defined,  for any $u\in \Z_+$,  by 
\[
\tau^u: \C[x, y]\longrightarrow \C[x, y], \quad
x\mapsto x+y^u, \
y\mapsto y,  \ \text{   where $y^u=1$ if  $u=0$}.
\]

The following theorem is a classical result.  Different proofs of it were given by various people,  e.g., Nagata \cite{N}, 
McKay and Wang  \cite{MW}.
\begin{theorem}\label{thm:Nagata}
$Aut(\C[x, y])$ is generated by $\sigma, \tau_c, \tau^u$ for all $c\in \C^*, u\in \Z_+$.
\end{theorem}
This result is closely related to the celebrated Jacobian conjecture.
Given any elements $f, g$ in the algebra $k[x, y]$ of polynomials in two variables over a field $k$ of characteristic $0$, denote $J(f, g)=\det \begin{pmatrix} \begin{array}{cc}\frac{\partial{f}}{\partial{x}} &\frac{\partial{f}}{\partial{y}}\\\frac{\partial{g}}{\partial{x}} &\frac{\partial{g}}{\partial{y}}\end{array}\end{pmatrix}$. Call $(f, g)$ a Jacobian pair if the Jacobian condition  $J(f, g)\in \C^*$ is satisfied.
The Jacobian conjecture in the present case  states that  $k[x,y]=k[f,g]$ if $(f, g)$ is a Jacobian pair. In this case,
the map $x\mapsto f$, $y\mapsto g$ is an automorphism of $k[x,y]$, and hence the relevance of Theorem \ref{thm:Nagata} when $k=\C$.

\begin{remark}
In sharp contrast to Corollary \ref{cor:auto-Laurent} for the Laurent polynomial algebra,  the automorphism group for the polynomial algebra $\C[x_1, x_2, \cdots, x_m]$ is much harder to determine.
\end{remark}

\subsection{$\Z$-gradings of $\cA_q(2)$}\label{sect:grading}
%\subsection{The quantum coordinate algebra of $\C^2$}
%
%

If $H$ is a non-degenerate $2\times 2$ integral matrix, we take $H=S$ without losing generality. The corresponding algebra $\cA_q(2)=\cA_q[H]$ is an associative algebra generated by $x, y$ subject to the relation
\beq\label{eq:x-y}
 xy=qyx,
\eeq
where $q\in \C^*$ is a fixed parameter. This
is the quantum coordinate algebra of the natural $\U_q(\fsl_2)$-module $\C^2$.
The set of ordered monomials $\{x^my^n\mid m, n\in\Z_+\}$ is a basis of  $\cA_q(2)$.

In this case,  $\cL_q(2)=\cL_q[H]$ is the QLPA generated by
$x, y$ and their inverses $x^{-1}, y^{-1}$ subject to the relation \eqref{eq:x-y}. It is also known as the function algebra of the quantum torus of rank $2$.  Clearly $\{x^my^n\mid m, n\in\Z\}$ is a basis. 

Note that  $\cL_q(2)$ can be regarded as the localisation of $\cA_q(2)$ at $\{ x^k y^\ell\mid k, \ell\in\Z_+\}$, which is well defined since $x, y$ quasi commute. 
It is easy to see  \cite{BEG} that $\cL_q(2)$ is simple at generic $q$
 in the sense that it contains no non-trivial $2$-sided ideals.

%\section{Irreducible modules}

%%%%%%%%%%%%%%%%%%%%%%%%%%%%%%%%%%
%\iffalse
%%%%%%%%%%%%%%%%%%%%%%%%%%%%%%%%%%
\subsubsection{Twisted derivations}\label{sect:q-deriv}
Denote the quantum integers by $[m]_q=\frac{q^m-1}{q-1}$ for $m\in\Z$. Consider the following linear difference operators $D_x, D_y: \cA_q(2)\longrightarrow \cA_q(2)$, and grading maps $\tau_x, \tau_y: \cA_q(2)\longrightarrow \cA_q(2)$, defined by
\[
\begin{aligned}
&D_x(x^ay^b)=[a]_q x^{a-1}y^b, \quad D_y(y^b x^a)=[b]_q y^{b-1} x^{a},\\
&\tau_x(x^ay^b)=q^ax^ay^b, \quad \tau_y(x^ay^b)=q^bx^ay^b,\quad \forall a, b\in \Z_+.
\end{aligned}
\]
It is routine to check the following modified Leibniz rule.
\begin{lemma} The grading operators 
$\tau_y$ and $\tau_y$ satisfy
\[
\tau_x(f g)=\tau_x(f)\tau_x(g), \quad \tau_y(f g)=\tau_y(f)\tau_y(g), 
\]
and the difference operators $D_x, D_y$ obey the following
twisted Leibniz rules
\[
\begin{aligned}
&D_x(fg)=D_x(f)\tau_x(g)+\tau_y^{-1}(f)D_x(g), \\
&D_y(fg)=D_y(f)\tau_y(g)+\tau_x(f)D_y(g), 
%&\tau_x(f g)=\tau_x(f)\tau_x(g), \quad \tau_y(f g)=\tau_y(f)\tau_y(g).
\end{aligned}
\]
for all $f, g\in \cA_q(2)$. 
\end{lemma}
Accordingly, $D_x$ is called a $(\tau_x, \tau_y^{-1})$-derivation, and $D_y$ a $(\tau_x^{-1}, \tau_y)$-derivation.

One can extend the domain of definitions of these operators to $\cL_q(2)$. Then  
\[
D_x = \frac{x^{-1}}{q-1}(\tau_x - 1), \quad D_y = \frac{y^{-1}}{q-1}(\tau_y - 1). 
\]
These $q$-difference operators and slight variations of them have long been studied \cite{J}, and play an important role 
in the theory of quantum groups (see, e.g., \cite{BMZG, M}).

\subsubsection{$\Z$-gradings}\label{sect:grading}
%
%
%%%%%%%%%%%%%%%%%%%%%%%%%%%%%%%%%%%%%%%%
%\fi
%%%%%%%%%%%%%%%%%%%%%%%%%%%%%%%%%%%%%%%

The algebra $\cA_q(2)$ can be equipped with various $\Z$-gradings such that  the degree zero component is a polynomial algebra in one variable.

The natural $\Z$-grading is defined by letting $q^D=\tau_x^{-1}\tau_y$, and requiring
\beq\label{eq:q-deg}
q^D(u) = q^{deg(u)} u, \quad \text{for any homogeneous $u\in\cA_q(2)$}, 
\eeq
where $deg(u)$ is the degree of $u$. 
Then $deg(x)=-deg(y)=-1$. The degree zero subalgebra $\cA_q(2)_0$ is the linear span of $\{x^j y^j\mid j\in \Z_+\}$,  which is isomorphic to the algebra of polynomials in $xy$. For all $m\ge 1$, the homogeneous components of degreees $\pm m$ are respectively given by $\cA_q(2)_{-m}=\cA_q(2)_0 x^m$ and $\cA_q(2)_{m}=\cA_q(2)_0 y^m$.

\medskip

To consider other  $\Z$-gradings, we note the following fact.
\begin{fact}\label{fact}
If $a, b$ are a pair of co-prime positive integers, there exist $u, v\in \Z_+$ such that $a u -b v=1$ with $0\le v\le a$ and $0\le u\le b$, where  least one of the inequalities is strict. Also, $(b-u)a-(a-v)b=-1$.
\end{fact}
Corresponding to each pair of co-prime positive integers $a, b$, we define a  weighted $\Z$-grading by taking $q^D=\tau_x^{-b}\tau_y^a$ and imposing \eqref{eq:q-deg}.  Then 
$deg(x)=-b$ and $deg(y)=a$. With respect to this grading, the degree $0$ component $\cA_q(2)_0$ is the polynomial algebra in $x^a y^b$ with a basis $\{(x^ay^b)^k\mid k\in\Z_+\}$. Now $deg(x^c y^d)=wt(c, d):=\det\begin{pmatrix}a & b\\ c & d \end{pmatrix}$ for any $c, d\in\Z_+$.

Given any $A=\begin{pmatrix} a& b \\ c & d \end{pmatrix}\in M_2(\Z_+)$,  let
$
X_A= x^a y^b$ and $Y_A= x^c y^d.
$
By using the obvious fact that $A^T S A= \det(A) S$, we obtain
\[
X_A Y_A = q^{\det(A)} Y_A X_A.
\]
This easily leads to the following result. 
\begin{lemma} \label{lem:step-ops}
Retain notation of Fact \ref{fact}. The elements
\beq\label{eq:ladder}
K=x^a y^b,  \quad \wh{X}=x^{a-v} y^{b-u}, \quad \wh{Y}=x^v y^u,
\eeq
of $\cA_q(2)$ satisfy the following relations
\[
\begin{aligned}
&\wh{X} \wh{Y} = q^{-v(b-u)} K, \quad \wh{Y} \wh{X}= q^{-u(a-v)}  K, \quad \wh{X} \wh{Y} = q  \wh{Y} \wh{X}, \\
& K \wh{X} = q^{-1} \wh{X} K, \quad K \wh{Y} = q \wh{Y} K.
\end{aligned}
\]
\end{lemma}

The element $K$ plays the role of the grading operator for this weighted $\Z$-grading in the sense that
\[
K x^c y^d = q^{wt(c, d)} x^c y^d K.
\]

\brmk \label{rmk:eq-deg}\label{def:localised}
Two monomials $x^c y^d$ and $x^{c'} y^{d'}$ have the same degree  if and only if $\det\begin{pmatrix}a & b\\ c-c' & d-d' \end{pmatrix}=0$, that is, there exists $k\in\Z$ such that $\begin{pmatrix}c\\ d\end{pmatrix}-\begin{pmatrix}c'\\ d'\end{pmatrix}=k\begin{pmatrix}a\\ b\end{pmatrix}$.
\ermk

\begin{definition} For any pair $(a, b)\ne (0, 0)$ of non-negative integers, let $\wt{\cA}^{(a, b)}_q$ be the localisation of $\cA_q(2)$ at the multiplicative set $\{1, K, K^2, \dots \}$ where $K=x^a y^b$.
\end{definition}
Since $K$ quasi-commutes with all monomials in $x$ and $y$, the localisation is well defined. We can regard $\wt{\cA}^{(a, b)}_q$ as generated by the generators $x, y, K^{-1}$ satisfying
\[
\baln
&x y = q y x, \quad K=x^a y^b,  \quad K^{-1} K = K K^{-1}=1, \\
&K^{-1} x = q^b x K^{-1}, \quad K^{-1} y = q^{-a} y K^{-1}.
\ealn
\]
The following fact is quite obvious in view of Remark \ref{rmk:eq-deg}.
\begin{lemma}\label{lem:localised}
The elements $K, \wh{X}, \wh{Y}$ and $K^{-1}$ generate $\wt{\cA}^{(a, b)}_q$.
\end{lemma}

\section{The category of admissible $\cA_q(2)$-modules}

We introduce a category of admissible $\cA_q(2)$-modules, which have finite covers by submodules with natural locally finiteness properties and satisfy certain condition under localisation.  
The blocks of the category are determined, and simple objects are classified.

\subsection{Simple $\cA_q(2)$-modules}\label{sect:simples}
An $\cA_q(2)$-module $M$ is said to be locally finite with respect to $D=x^m y^n$ for  given $(m, n)\in\Z_+^2/\{(0, 0)\}$ if $\C[D] w$ is finite dimensional for any $ w\in M$. We denote by $\wt\cC_{(m, n)}$ the category of  locally $D$-finite $\cA_q(2)$-modules.

Let $(a, b)=(\frac{m}{\ell}, \frac{n}{\ell})$ with $\ell=gcd(m, n)$ being the largest common divisor of $m$ and $n$. We note in particular that  $(a, b)=(1, 0)$ or $(0, 1)$ if $m n=0$.  Now $D= q^{a b\ell(\ell-1)/2} K^\ell$ with $K=x^a y^b$ as defined by \eqref{eq:ladder}, thus $\wt\cC_{(m, n)}=\wt\cC_{(a, b)}$.

We now classify the simple modules in $\wt\cC_{(m, n)}$.

Assume that $M$ is a simple $\cA_q(2)$-module.  For any non-zero vector $w\in M$, there is an eigenvector $v\in \C[K] w$ of $K$ with some eigenvalue $\lambda$.
Then $v$ cyclically generates $M$.

Case $(1)$:  $\lambda=0$. We claim that there exist some scalars $\mu_1$ and $\mu_2$ such that the corresponding representation is given by
\beq\label{eq:1-dim}
x\mapsto \mu_1, \quad y\mapsto \mu_2, \quad \text{where $\mu_1\mu_2=0$}.
\eeq
To prove this, note that $v$ satisfies either $y^b v=0$ or $y^b v\ne 0$. If $y^b v=0$, there exists a non-negative integer $b_0< b$ such that $y^{b_0}v\ne 0$ but $y^{b_0+1}v=0$. Now
$y^{b_0}v$ generates $M$, hence $y M=0$. The simplicity of $M$ further requires that $ y^{b_0}v$ is an eigenvector of $x$ and $M=\C y^{b_0}v$. If $y^b v\ne 0$, there exists a non-negative integer $a_0< a $ such that $x^{a_0}y^b v\ne 0$ but $x^{a_0+1}y^b v=0$. Now $x^{a_0}y^b v$ generates $M$, thus $x M=0$. Arguing as in the previous case, we can also show that $x^{a_0}y^b v$ is an eigenvector of $y$ and $M=\C x^{a_0}y^b v$.

Case $(1')$: $\lambda\ne 0$ with $a b=0$. Now $(a, b)$ is either $(1,0)$ or $(0, 1)$.  In the first case $K=x$, and $M$ is spanned by $y^i v$ for $i\in\Z_+$. Such an $M$ is not simple unless $y v=0$. Similarly in the second case, $K=y$, and we have $x v=0$. Hence the module $M$ is as that studied in Case $(1)$.

Case $(2)$: $\lambda\ne 0$ and $a b\ne 0$. We have
\[
K x^cy^d v=q^{wt(c, d)}\lambda x^c y^d  v, \quad \forall c, d\in\Z_+.
\]
Thus $K$ is a semi-simple automorphism of $M$ as vector space. It follows Remark \ref{rmk:eq-deg} that if $x^cy^d v$ and $x^{c'} y^{d'} v$ have the same $K$ eigenvalue, then $x^cy^d v= \lambda^k x^{c'} y^{d'} v$ for some $k\in\Z$. Thus eigenspaces of $K$ are all $1$-dimensional.  Now we make use of Lemma \ref{lem:step-ops} with the pair $a, b$. We have
\[
\begin{aligned}
K \wh{Y} v= q\lambda  \wh{Y} v, \quad K \wh{X}  v= q^{-1}\lambda  \wh{X}  v.
\end{aligned}
\]
Thus the operator $\wh{Y}$ raises the eigenvalue by a factor $q$,  while $\wh{X}$ lowers the eigenvalue of $K$ by $q^{-1}$. This shows that $M$ has a basis $\{\wh{X}^\ell  v,  \wh{Y}^{\ell+1}  v\mid \ell\in\Z_+\}$. In view of Lemma \ref{lem:localised},  one can lift $M$ to an $\wt{\cA}^{(a, b)}_q$-module with $K^{-1} x^cy^d v=q^{-wt(c, d)}\lambda^{-1} x^c y^d  v$.

The structure of the simple module $M$ discussed above shows that 
a simple module cyclically generated by an eigenvector of $K$ with eigenvalue $\lambda'\ne  0$ is isomorphic to $M$ if an only if $\lambda'= q^k \lambda$ for some $k\in\Z$, that is, $\lambda'$ and $\lambda$ belong to the same equivalence class in ${\rm E}_q:=\C^*/q^\Z$, where $q^\Z=\{q^m\mid m\in\Z\}$. Therefore, locally $\C[K]$-finite simple modules are parametrised by points of the elliptic curve ${\rm E}_q$.

\begin{definition}
Given a pair of co-prime positive integers $(a, b)$ and $[\lambda]\in \C^*/q^\Z$, denote by $L_{(a, b)}[\lambda]$ the simple $\cA_q(2)$-module cyclically generated by an eigenvector of $K=x^a y^b$ with eigenvalue $\lambda$.
\end{definition}

\begin{lemma}\label{lem:iso-class}
Simple $\cA_q(2)$-modules $L_{(a, b)}[\lambda]$ and $L_{(a', b')}[\lambda']$ are isomorphic if and only if $(a, b)=(a', b')$ and $[\lambda]=[\lambda']$.
\end{lemma}
\begin{proof}
We have already shown that $L_{(a, b)}[\lambda]\not\cong L_{(a, b)}[\lambda']$ if $[\lambda]\ne [\lambda']$.
Assume $(a', b')\ne (a, b)$, and let $K'=x^{a'}y^{b'}$.
Let $v$ be an eigenvector of $K=x^a y^b$ which cyclically generates $L_{(a, b)}[\lambda]$.  For any $\ell\in\Z_+$, we have  $K {K'}^\ell v =q^{d\ell} {K'}^\ell v$ with $d=\det\bmtx a & b\\ a' & b'\emtx\ne 0$, and hence ${K'}^\ell v$ is a non-zero scalar multiple of $\wh{X}^{|d |\ell}v$ or $\wh{Y}^{|d| \ell}v$. Therefore $L_{(a, b)}[\lambda]$ is not locally finite with respect to $K'$, thus it can not be a homomorphic image of $L_{(a', b')}[\lambda']$.
\end{proof}

To summarise,
\begin{theorem}\label{thm:cat-lf}
Given any $(m, n)\in\Z_+^2/\{(0, 0)\}$, let $D=x^m y^n$. Then a simple module $L\in\wt\cC_{(m, n)}$ is isomorphic to
\begin{enumerate}[i)]
\item a $1$-dimensional module described by \eqref{eq:1-dim} if $\ker(D)\ne \{0\}$;

\item $L_{(a, b)}[\lambda]$ for some $[\lambda]\in \C^*/q^\Z$ if $\ker(D)=\{0\}$, where
$(a, b)=(\frac{m}{\ell}, \frac{n}{\ell})$ with $\ell=gcd(m, n)$.
\end{enumerate}
In particular,  $\dim L=1$ if  $m n=0$.
\end{theorem}

Let $\underline{c}=(c_1, c_2)^T\in{\C^*}^2$,  and denote by $\tau_{\underline{c}}$ the corresponding automorphism of $\cA_q(2)$. We twist the simple $\cA_q(2)$-module $L_{(a, b)}[1]$ by $\tau_{\underline{c}}$, and denote the resulting module by $L^{\tau_{\underline{c}}}_{(a, b)}[1]$. Now any $u\in \cA_q(2)$ acts on $L^{\tau_{\underline{c}}}_{(a, b)}[1]$ by
\beq
u\star w = \tau_{\underline{c}}(u) w, \quad  \forall w\in M.
\eeq
Note in particular that $K\star v=   c_1^a c_2^b v$,  thus we have the following result.

\begin{lemma}\label{lem:L-tau}
For any $\underline{c}=(c_1, c_2)^T\in{\C^*}^2$, there is the following isomorphism 
$
L^{\tau_{\underline{c}}}_{(a, b)}[1]\cong L_{(a, b)}[c_1^a c_2^b]
$
of $\cA_q(2)$-modules. 
\end{lemma}

Two explicit constructions of the simple $\cA_q(2)$-modules $L_{(a, b)}[\lambda]$ will be given in Section \ref{sect:constructions} using very different methodologies.

\subsection{The category of admissible modules}

Recall from Definition \ref{def:localised} the localisation  
$\wt{\cA}^{(a, b)}_q$ of $\cA_q(2)$ for any pair of non-negative integers $(a, b)\ne (0, 0)$. We have the  canonical embedding $\cA_q(2)\longhookrightarrow\wt{\cA}^{(a, b)}_q$. Given any $\cA_q(2)$-module $M$, one can localise it to an $\wt{\cA}^{(a, b)}_q$-module  
$
\wt{M}^{(a, b)}:= \wt{\cA}^{(a, b)}_q\otimes_{\cA_q(2)} M. 
$
There exists the canonical $\cA_q(2)$-module homomorphism 
\beq\label{eq:local-hom}
\kappa^{(a, b)}: M \lra \wt{M}^{(a, b)}, \quad v\mapsto 1\otimes v, \quad \forall v\in M.
\eeq
Note that $v\in \ker(\kappa^{(a, b)})$ if and only if $K^\ell v=0$ for some $\ell\ge 1$, where $K=x^a y^b$.

We now introduce the notion of admissible modules for $\cA_q(2)$.  

\begin{definition}
An $\cA_q(2)$-module $M$ is called admissible if there are submodules $M_j\subset M$ with $1\le j\le m_M<\infty$, which cover $M$ in the sense that $M=\sum_{j}M_j$,  and satisfy the following conditions:  
\begin{enumerate}[i)]
\item 
$M_j \in\wt\cC_{(m_j, n_j)}$ for $(m_j, n_j)\in\Z_+^2/\{(0, 0)\}$, and 
\item the canonical map $\kappa^{(1, 1)}:  M\lra\wt{M}^{(1, 1)}$ defined by \eqref{eq:local-hom}  is injective. 
\end{enumerate}
Denote by $\cC$ the category of admissible $\cA_q(2)$-modules.
\end{definition}

\begin{remark} 
\begin{enumerate}[1)]
\item 
One may replace $\wt{\cA}^{(1, 1)}_q$  in 
condition $ii)$ by $\wt{\cA}^{(a, b)}_q$ for any positive integers $a, b$.  

\item Finite dimensional $\cA_q(2)$-modules, and  in particular the $1$-dimensional modules described by \eqref{eq:1-dim}, are not admissible modules. 

%%%%%%%%%%%%%%%%%%%%%%
\iffalse
%%%%%%%%%%%%%%%%%%%%
\item 
%The category $\cC$ is not closed under taking quotients because of the localisation condition. 

Let $M=\C[t]$, which has an $\cA_q(2)$-module structure defined by
\[
y f(t)= t f(t), \quad x f(t) = f(q t), \quad \forall f(t)\in \C[t].
\]
Clearly $M$ belongs to $\cC_{(1, 0)}\subset \cC$, and so do also its submodules $M_\ell = t^\ell M$, but the quotient modules $M/M_\ell$ for all $\ell\ge 1$  violate condition $ii)$ of the definition.
%%%%%%%%%%%%%%%%%%%%
\fi
%%%%%%%%%%%%%%%%%%%

\end{enumerate}
\end{remark}

We say that $(a, b)\in\Z_+^2/\{(0, 0)\}$ is co-prime if the greatest common divisor of $a$ and $b$ is $1$. This has  the usual meaning when $a, b$ are both non-zero. If one of them is zero, then $(a, b)$ is either $(1,0)$ or $(0, 1)$. Let
\[
\cP=\{(a, b)\in\Z_+^2/\{(0, 0)\} \mid (a, b) \text{ is co-prime}\}, \quad \cP_{red}=\{(a, b)\in \cP\mid a b\ne 0\}.
\]
Denote by $\cC_{(m, n)}$ the full subcategory of $\wt\cC_{(m, n)}$ with admissible objects.
Then each $\cC_{(m, n)}$ is a full subcategory of $\cC$. 

We have the following result.
\begin{lemma}  \label{lem:cat-adm}
The category $\cC$ of admissible $\cA_q(2)$-modules can be expressed as
\[
\cC=\bigoplus_{(a, b)\in\cP} \cC_{(a, b)},
\]
that is,  the objects of $\cC$ are finite direct sums of objects of the full subcategories $\cC_{(a, b)}$, and all homomorphisms between objects belonging to different subcategories are zero.
\begin{proof}
We first show that for any $M\in \cC_{(a, b)}$ and $M'\in \cC_{(a', b')}$,
\beq\label{eq:no-morph}
\Hom_{\cA_q(2)}(M, M')=0, \ \text{ if $(a, b)\ne (a', b')$}.
\eeq
The general idea of the proof is the same as that for Lemma \ref{lem:iso-class}.
Let $K=x^a y^b$ and $K'=x^{a'} y^{b'}$. Then
$K K' = q^d K' K$ with $d=\det\bmtx a & b\\ a'& b' \emtx\ne 0$. For any $w\ne 0$ in $M$, if $w':=\varphi(w)\ne 0$, then it is locally $K$-finite. Thus there exists an eigenvector $v' \in \C[K] w'$ of $K$ with some eigenvalue $\mu$. We claim that $\mu\ne 0$, as otherwise $K v' = x^{a} y^b v'=0$, which contradicts condition $ii)$ for admissibility of $M'$. To see this,  note that if $a\le b$, then $x^{b-a} K v' = x^b y^b v' = q^{b(b-1)/2} (x y)^b v'=0$, which implies that $1\otimes v'$ is a zero vector in $\cA_q^{(1,1)}\otimes_{\cA_q(2)} M'$. Clearly the $a\ge b$ case is the same. 
Now $K v'= \lambda v'$ leads to
\[
K {K'}^{\ell} v' = q^{-d\ell} \lambda {K'}^{\ell} v', \quad \forall \ell\in \Z_+.
\]
Thus the ${K'}^{\ell} v'$ are linearly independent, as otherwise there is some positive integer $\ell_0$ such that ${K'}^{\ell_0} v'=0$, which would contradict admissibility of $M'$. 

Given $M\in \cC$, we let $(a_i, b_i)\in\Z_+^2/\{(0, 0)\}$ with $i=1, 2, \dots, r$ be  pairs of non-negative integers such that there exist submodules $M_i\in\cC_{(a,_i, b_i)}$ covering $M$, that is, $M=M_1+M_2+\dots+M_r$. 
We can assume that each $(a_i, b_i)$ is co-prime (since $\cC_{(a_i, b_i)}=\cC_{\ell(a_i, b_i)}$ for all positive integer $\ell$), and $(a_i, b_i)\ne (a_j, b_j)$ for all $i\ne j$. It follows \eqref{eq:no-morph} that $M$ is the direct sum of the submodules $M_i$.
\end{proof}
\end{lemma}

\begin{corollary} \label{cor:simple-C}
The set of simple objects of $\cC$ is
\[
\left\{L_{(a, b)}[\lambda]\mid (a, b)\in\cP_{red}, \ [\lambda]\in\C^*/q^\Z\right\}.
 \]
\end{corollary}
\begin{proof}
It follows Lemma \ref{lem:cat-adm}  that each simple object of $\cC$ belongs to $\cC_{(a, b)}$ for some $(a, b)\in\cP$.  By Theorem \ref{thm:cat-lf}, the simple objects of $\cC_{(a, b)}$ are $L_{(a, b)}[\lambda]$ with $\lambda\in\C^*/q^\Z$ if $(a, b)\in\cP_{red}$, and $\cC_{(a, b)}$ has no simple objects if $a b=0$.
\end{proof}

Call an object $M\in \cC_{(a, b)}$ a weight module if $K=x^a y^b$ acts semi-simply. If  $v\in M$ is an eigenvector  of $K$ with eigenvalue $\lambda$, we call it a weight vector of weight $\lambda$. We denote by $\cC_{(a, b)}^{wt}$ the full subcategory of $\cC_{(a, b)}$ consisting of weight modules.

\begin{lemma} The following full subcategory of $\cC$ is semi-simple.
\[
C_{red}^{wt}:=\bigoplus_{(a, b)\in\cP_{red}} \cC_{(a, b)}^{wt}.
\]
\end{lemma}
\begin{proof}
Consider any object $M$ in $\cC_{(a, b)}^{wt}$ for $(a, b)\in\cP_{red}$. Given any two weight vectors $v, w\in M$ of weight $\mu$ and $\nu$ respectively, we claim that if they are not contained in one simple submodule, then they belong to the direct sum of two simple submodules. This implies that $M$ is semi-simple.

The claim is easy to see. Since $\cA_q(2) v = L_{(a, b)}[\mu]$ and $\cA_q(2) w = L_{(a, b)}[\nu]$ are both simple submodules,  either $\cA_q(2) v = \cA_q(2) w$ or $\cA_q(2) v \cap \cA_q(2) w=\{0\}$. This proves the claim,   completing the proof of the lemma.
\end{proof}

\section{Constructions of simple admissible $\cA_q(2)$-modules}\label{sect:constructions}

We give two explicit constructions of the simple admissible $\cA_q(2)$-modules.  One construction produces simple $\cA_q(2)$-modules from $\cL_q(2)$-modules via monomorphisms
$\cA_q(2)\stackrel{\iota}{\longhookrightarrow} \cL_q(2)\stackrel{\tau}{\longrightarrow} \cL_q(2)$ with $\iota$ being the natural embedding of $\cA_q(2)$ in $\cL_q(2)$, and $\tau$ being different automorphisms of $\cL_q(2)$.  
The other construction makes use of holonomic $\SD_q$-modules for the algebra of $q$-difference operators $\SD_q$.

\subsection{Construction by twisted $\cL_q(2)$-modules}\label{sect:restrict}

Retain the notation in Section \ref{sect:grading}.
Consider the following $\cL_q(2)$-action on $V:=\C[t, t^{-1}]$.
\[
x f(t) = f(q t), \quad y f(t) = t f(t), \quad f(t)\in V.
\]
This makes $V$ a simple $\cL_q(2)$-module, which will be called the basic module.
To see the simplicity, note that $x$ is the grading operator such that each monomial $t^m$ is an eigenvector of $x$ with eigenvalue $q^m$. For $q$ being generic, $q^m\ne q^n$ if $m\ne n$. The operators $y$ and $y^{-1}$ are the raising and lowering operators of step one. Hence the basic module is irreducible.

We can twist the basic $\cL_q(2)$-module by any $\cL_q(2)$-automorphism. For a given automorphism $\tau=\tau_{\underline{c}}\tau_A$ with $A=\begin{pmatrix}a_{11}& a_{12}\\a_{21} & a_{22} \end{pmatrix}$ in ${\rm SL}_2(\Z)$ and $\underline{c}=(c_1, c_2)^T\in{\C^*}^2$, where we recall that
\[
\tau(x)=c_1^{a_{11}}c_2^{a_{21}}x^{a_{11}}y^{a_{21}}, \quad \tau(y)=c_1^{a_{12}} c_2^{a_{22}} x^{a_{12}}y^{a_{22}},
\]
 the twisted
$\cL_q(2)$-action on $\C[t, t^{-1}]$ is defined by
\beq
x\star f(t) = \tau(x) f(t), \quad y \star  f(t) = \tau(y) f(t), \quad f(t)\in \C[t, t^{-1}].
\eeq
Note in particular that for all $t^m$ with $m\in\Z$, we have
\beq\label{eq:tw-act}
x\star t^m=c_1^{a_{11}}c_2^{a_{21}} q^{a_{11}(m+a_{21})}t^{m+a_{21}}, \quad y\star t^m=c_1^{a_{12}} c_2^{a_{22}} q^{a_{12}(m+a_{22})}t^{m+a_{22}}.
\eeq
We denote this twisted $\cL_q(2)$-module by $V^{\tau}$, which remains simple.

Now we consider restrictions of twisted $\cL_q(2)$-modules to  modules for the subalgebra $\cA_q(2)$. Restriction  rarely takes simples to simples.
However, we have the following pleasant result.

\begin{theorem} \label{thm:restr}
Retain notation above.
The twisted $\cL_q(2)$-module $V^{\tau}$ restricts to a simple $\cA_q(2)$-module $V^{\tau}|_{\cA_q(2)}$ if and only if $a_{21}a_{22}<0$. In this case, $V^{\tau}|_{\cA_q(2)}$ is isomorphism to the simple $\cA_q(2)$-module $L_{(a, b)}[c_1^\varepsilon]$ with $(a, b)=(|a_{22}|, |a_{21}|)$ and $\varepsilon=sign(a_{22})=\pm 1$.
\end{theorem}

\begin{proof}
If $a_{21}a_{22}\ge 0$, then either $a_{21}\ge 0, a_{22}\ge 0$ or $a_{21}\le 0, a_{22}\le 0$. In the former case,  the subspaces $t^m\C[t]$ are proper $\cA_q(2)$ submodules for any $m\in \Z$, and in the latter case, $t^m\C[t^{-1}]$ are submodules.

If $a_{21}a_{22}<0$, one can immediately see from \eqref{eq:tw-act} that none of the subspaces  $t^m\C[t]$ and $t^m\C[t^{-1}]$ is an $\cA_q(2)$-submodule. In this case, $(a, b)=(|a_{22}|, |a_{21}|)$ is a pair of co-prime positive integers since $\det(A)=1$,  thus by Lemma \ref{lem:step-ops}, we have the elements $K=x^ay^b$, $\wh{X}$ and $\wh{Y}$ in the $\cA_q(2)$  subalgebra of $\cL_q(2)$. Now $K$ acts on $V^{\tau}$ semi-simply with non-zero eigenvalues,  and $\wh{Y}$ and $\wh{X}$ are step-$1$ raising and lowering operators. Hence $V^{\tau}|_{\cA_q(2)}$ is a simple $\cA_q(2)$-module.

From \eqref{eq:tw-act} we can easily deduce that for any positive integer $k$,
\[
\baln
x^k\star t^m&= \left(c_1^{a_{11}}c_2^{a_{21}}\right)^k q^{a_{11}\left(k m + \frac{k(k+1)}2 a_{21}\right)} t^{m+k a_{21}},  \\
y^k\star t^m&= \left(c_1^{a_{12}} c_2^{a_{22}}\right)^k q^{a_{12}\left(k m + \frac{k(k+1)}2 a_{22}\right)} t^{m+k a_{22}}.
\ealn
\]
Using these formulae, we can easily show that 
\[
\baln
K\star1%&=\left(c_1^{a_{12}} c_2^{a_{22}}\right)^{b} q^{\frac{b(b+1)}2 a_{12}a_{22}} x^a\star t^{b a_{22}}\\
	%&=\left(c_1^{a_{11}}c_2^{a_{21}}\right)^a  \left(c_1^{a_{12}} c_2^{a_{22}}\right)^{b} 
	%q^{\frac{b(b+1)}2 a_{12}a_{22}+a_{11}\left(a b a_{22}+\frac{a(a+1)}2 a_{21} \right)}  t^{a a_{21}+ b a_{22}}\\
	&=c_1^{\varepsilon} q^{\frac{b(b+1)}2 a_{12}a_{22}+a_{11}\left(a b a_{22}+\frac{a(a+1)}2 a_{21} \right)}.
\ealn
\]
Hence
$V^{\tau}|_{\cA_q(2)}\cong L_{(a, b)}[c_1^{\varepsilon}]$.
\end{proof}

The following result immediately follows from Corollary \ref{cor:simple-C} and Theorem \ref{thm:restr}.
\begin{corollary}\label{cor:full-set}
Every simple object of $\cC$ can be obtained as the restriction of the basic
$\cL_q(2)$-module twisted by an appropriate automorphism.
\end{corollary}

\subsection{The algebra $\SD_q$ of $q$-difference operators}
We present another construction of the simple admissible $\cA_q(2)$-modules using the theory of
$q$-difference modules (see, e.g., \cite{dV, RSZ}), which is a $q$-difference analogue of $\SD$-modules theory.
There is extensive literature on $q$-difference modules and  $q$-difference equations (see \cite{RSZ} for a review and reference therein). 

The proof of the main result of this section, Theorem \ref{thm:construct-d-diff},  requires some elementary facts on $q$-difference modules, which can all be extracted from \cite{RSZ}.  We provide detailed proofs for them to make the paper reader friendly, claiming no originality.

We start by enlarging $\cL_q(2)$ to the algebra $\SD_q$ of $q$-difference operators.  

For any indeterminate $t$,
we denote by $\C[[t]]$ the ring of formal power series in $t$, which is the adic completion of $\C[t]$ with respect to the maximal ideal $t\C[t]$. Let $\C((t))$ be the fraction field of $\C[[t]]$, which is the ring of formal Laurent power series  $\C((t))=\sum_{k\in\Z_+} t^{-k}\C[[t]]$.

The underlying $\C$-vector space  of $\SD_q$ is $\C((x))[y, y^{-1}]$, and the defining relation of the algebra is still
$x y = q y x$, or more generally
$
y f(x)  =f (q^{-1} x) y$ for all $f(x)\in\C((x)).
$
Call $\SD_q$ the algebra of $q$-difference operators of dimension $1$.

\begin{remark}
We have chosen to complete $\C[x]$ instead of $\C[y]$ for the psychological reason that we are more used to thinking about ``functions of $x$''. 
\end{remark}

Given any element $A\in\SD_q$, there exists a smallest $k_0>0$ such that
$y^{k_0}A=\sum_{i=0}^n b_{n-i} y^i$ for some finite $n$ with $b_i\in\C((x))$ such that $b_0\ne 0\ne b_n$. Hence
\[
A= u_A P_A,   \quad u_A=y^{-k_0}b_, \quad P_A=\sum_{i=0}^n a_{n-i} y^i, \quad a_i= b_i/b_0.
\]
Now $a_0=1$ and $a_n\ne 0$, thus we call $P_A$ a monic polynomial in $y$ over $\C((x))$ with non-zero ``constant'' term (which is an element in $\C((x))$). This factorisation of $A$ is unique.

We call a non-commutative ring a left (resp. right)  domain if there is no left (resp. right) zero divisor apart from $0$. A left (resp. right) non-commutative domain is principal if every left (resp. right)  ideal is generated by a single element.

\begin{lemma}\label{lem:Dq-structure}
The $\C$-algebra $\SD_q$ is a non-commutative left (resp. right)  principal ideal domain. Furthermore,
any left (resp. right) ideal is given by $\cJ_P:=\SD_q P$ (resp. $P\SD_q $) for some monic polynomial $P$ in $y$ over $\C((x))$ with non-zero ``constant'' term
 of the form
\begin{eqnarray}
&&P=y^n + a_1 y^{n-1} + a_2 y^{n-2} +\dots+ a_n, \quad a_i\in\C((x)), \ a_n\ne 0,   \label{eq:diff-op}\\
&&\text{with  $P=1$ if $n=0$}. \nonumber
\end{eqnarray}
\begin{proof}
We consider left ideals only, the proof for right ideals is similar. 
We have already seen that any element $A\in \SD_q$ can be uniquely factorised into $A=u_A P_A$, where $u_A$ is a unit and $P_A$ is a monic polynomial in $y$ over $\C((x))$ with non-zero ``constant'' term. For any element $Q$ of a left ideal $\cJ$, the corresponding monic polynomial $P_Q$ in $y$ belongs to $\cJ$. The lowest degree polynomial among $P_Q$ for all non-zero $Q\in\SD_q$ is unique, which generates $\cJ$. 
\end{proof}
\end{lemma}

\begin{corollary}
A left ideal $\cJ_P$ of $\SD_q$ is maximal if and only if $P\ne 1$ and is not the product of two monic polynomial in $y$ over $\C((x))$ of degrees $\ge 1$ with non-zero ``constant'' terms.
\end{corollary}
\begin{proof}
Let $P, Q$ be monic polynomials in $y$ over $\C((x))$ with non-zero ``constant'' terms. By Lemma \ref{lem:Dq-structure},  $\SD_q\ne \cJ_Q\supsetneq \cJ_P$ if and only if there exists some non-unit $P'\in\SD_q$ such that $P=P' Q$. This requires $P'$ be a monic polynomial in $y$ over $\C((x))$ of degree $\ge 1$ with non-zero ``constant'' term.
\end{proof}

\begin{example}\label{ex:max-ideal} 

\noindent (1).
%\begin{enumerate}
%\item
For any positive integer $\ell$ and $\lambda\in\C^*$,  let
$
%\beq
P= y^2 - \lambda x^{-2\ell-1}.
%\eeq
$
Then $\cJ_P=\SD_q P$ is a maximal left ideal.  Otherwise, there would exist $Q= y + a(x)$ and  $Q' = y + b(x)$ such that $Q Q'=P$, that is,
\[
a(x) + b(q^{-1} x) =0, \quad a(x) b(x) = - \lambda x^{-2\ell-1}.
\]
We can always write $a(x)=x^k \alpha(x)$ and $b(x)=x^{k'}\beta(x)$ for some integers $k, k'$,  and formal powers series $\alpha=\sum_{i\ge 0}\alpha_i x^i$ and $\beta=\sum_{i\ge 0}\beta_i x^i$ in $x$ with $\alpha_0\ne 0\ne \beta_0$.  Then the first relation in particular requires $k=k'$, which contradicts the second relation.
Hence $\cJ_P$ is a maximal left ideal.

\medskip
\noindent (2).
%\item
Given any  $\mu, \nu \in\C^*$,  positive integers $a, \ell$, and $0\ne f(x)\in \C[[x]]$, let
\[
P= y^{2} + \mu q^{a} x^{\ell -a} y + \nu q^{a } x^{-2a} (1 + x f(x)).
\]
We claim that $\cJ_P=\SD_q P$ is a maximal left ideal.
To prove this, let $K=x^a y$. Then we have 
$
%\[
P=q^{a}x^{-2a}\left(K^2 + \mu x^\ell K + \nu(1+ x f(x))\right). 
%\]
$
If $P$ factorises, there exist $\alpha(x), \beta(x)\in\C((x))$ such that
$
%\[
P=q^{a}x^{-2a} (K+\alpha)(K+\beta).
%\]
$
This leads to
\beq\label{eq:alpha-beta}
\alpha(x) + \beta(q^{-1} x)= \mu x^\ell, \quad \alpha(x) \beta(x) = \nu(1+ x f(x)).
\eeq
There always exist $k, k'\in\Z$ such that
$
%\[
\alpha(x)= x^k \sum_{i\ge 0}\alpha_i x^i,$  $\beta(x)=x^{k'}\sum_{i\ge 0}\beta_i x^i$ with $\alpha_0\ne 0\ne \beta_0.
%\]
$
The second equation in \eqref{eq:alpha-beta} requires in particular that $k=-k'$. This implies that the coefficient of $x^{-|k|}$ in $\alpha(x) + \beta(q^{-1} x)$ is either $\alpha_0$ or $q^k \beta_0$, thus is non-zero. This contradicts the first equation of \eqref{eq:alpha-beta}, proving the claim.
%\end{enumerate}
\end{example}

\begin{remark}\label{rmk:factor}
Factorisation of a polynomial  in $y$ over $\C((x))$ is not unique in general. For example,
$
\left(y - \frac{1- q^{-1} x}{1-x}\right)\left(y-x\right) =  \left(y- \frac{x- q^{-1}x^2}{1-x}\right)\left(y-1\right).
$
\end{remark}

%
%
%\subsection{Definition of $q$-difference modules}
\subsection{Construction by $\SD_q$-modules}

Introduce the following $\C$-linear map
\beq\label{eq:mod-1}
\varphi_q: \C((x))\lra \C((x)), \quad \varphi_q(f(x))= f(q^{-1} x),
\eeq
which is a ring automorphism. In particular,
$
\varphi_q(f g)= \varphi_q(f)\varphi_q(g)
$
for any $f, g\in \C((x))$. Now $(\C((x)), \varphi_q)$ is a $q$-difference algebra (see, e.g., \cite{dV, RSZ} for the definition).
A left $q$-difference module $M$ for the $q$-difference algebra $(\C((x)), \varphi_q)$ is a finite dimensional $\C((x))$-vector space equipped with a $\C$-linear automorphism $\Phi_q: M\longrightarrow M$ such that
\beq
\Phi_q(f(x) v) = \varphi_q(f(x)) \Phi_q(v), \quad \forall v\in M, \ f\in\C((x)).
\eeq
Call $M$ cyclically generated if there is a vector $v$ such that $M=\sum_{i\in\Z}\C((x)) \Phi_q^i(v)$.  A homomorphism from a $q$-difference module $(M, \Phi_q)$ to another $(N, \Psi_q)$ is a $C((x))$-linear map $f: M \lra N$ such that
$
f\circ \Phi_q = \Psi_q\circ f.
$

\begin{example} \label{eg:free-module}
An obvious $q$-difference module structure on $\C((x))^n$ is
\[
\Phi_q\begin{pmatrix}a_1(x) \\  a_2(x)\\ \vdots \\ a_n(x)\end{pmatrix} = \begin{pmatrix} \varphi_q(a_1(x)) \\  \varphi_q(a_2(x)) \\ \vdots \\ \varphi_q(a_n(x))\end{pmatrix} = \begin{pmatrix}a_1(q^{-1} x) \\  a_2(q^{-1} x)\\ \vdots \\ a_n(q^{-1} x)\end{pmatrix}, \quad \forall a_i\in\C((x)).
\]
\end{example}

A $q$-difference module $M$ for $(\C((x)), \varphi_q)$ naturally leads to a  left $\SD_q$-module with
\beq\label{eq:de-mod}
x.v = x v, \quad y.v = \Phi_q(v), \quad \forall v\in M,
\eeq
and a homomorphism of $q$-difference module gives rise to  a homomorphism of the corresponding $\SD_q$-modules. Clearly this also goes through in the reverse direction for finite dimensional (over $\C((x))$) $\SD_q$-modules.  It is easy to see that
the category of $q$-difference modules is isomorphic to the category of $\SD_q$-modules consisting of objects  which are finite dimensional over $\C((x))$.  Thus we will use the terms $q$-difference module and $\SD_q$-module interchangeably.

\begin{remark}
There is a notion of holonomic  $\SD_q$-modules. For a module $M$ generated by a finite dimensional $\C$-vector space $M_0$, let $M_N=F_N M_0$ where $F_N=\C\text{-span}\{x^i y^j\mid |i|+|j|\le N\}$. Let
 $d(M):=\lim_{N\to\infty}\frac{\log(\dim_\C M_N )}{\log N}$, which  is the Gelfand–Kirillov dimension of $M$ obviously satisfying the Bernstein inequality $1\le d(M)\le 2$. The module $M$ is holonomic if $d(M)=1$. Clearly $q$-difference modules are holonomic.
\end{remark}

Any finitely generated $\SD_q$-module is the image of a surjection $\pi: \SD_q^n\lra M$ for some finite $n$, and thus $M\cong \SD_q^n/ker(\pi)$.  In particular,  for any left ideal $\cJ$ of $\SD_q$, we have the left module $\SD_q/\cJ$.
\begin{remark}
\begin{enumerate}
\item
If $\cJ, \cJ'$ are proper left ideals such that $\cJ'\supsetneq \cJ$, then $\cJ'/\cJ$ is a proper submodule of $\SD_q/\cJ$.  Thus $\SD_q/\cJ\ne 0$ is simple if and only if $\cJ$ is maximal, i.e., $\cJ\ne \SD_q$ and is not contained in any proper left ideal.
\item
If $\cJ_1, \cJ_2, \cJ$ are distinct proper left ideals of $\SD_q$ such that
$
\cJ_1+\cJ_2=\SD_q$ and $\cJ_1\cap\cJ_2=\cJ,
$
then  $
%\[
\frac{\SD_q}\cJ=\frac{\cJ_1}\cJ \oplus \frac{\cJ_2}\cJ.
%\]
$
\end{enumerate}
\end{remark}

\begin{lemma}\label{lem:cyclic}
The cyclically generated $q$-difference modules are in bijection with the left $\SD_q$-modules $\SD_q/\cJ$ for left ideals $\cJ$.
\end{lemma}
\begin{proof}
For a left ideal  $\cJ_P$ with $P$ of the form \eqref{eq:diff-op},  the elements $\upsilon_i= y^i +\cJ_P$ for $i=0, 1, \dots, n-1$,  form a basis of $\SD_q/\cJ_P$ as a $\C((x))$-vector space.  We have 
\[
\begin{aligned}
& y  \upsilon_i = \upsilon_{i+1}, \quad y^{-1}  \upsilon_i = \upsilon_{i-1}, \quad  i<n-1, \quad\text{with}  \\
&\upsilon_n:= -\sum_{i=0}^{n-1} a_{n-i}\upsilon_i,\quad
 \upsilon_{-1} = - \sum_{i=0}^{n-1} \frac{a_{n-1-i}(q x)}{a_n(q x)} \upsilon_i, \  \text{ $a_0=1$}.
\end{aligned}
\]
Now $\SD_q/\cJ_P$ becomes a $q$-difference module with the structure map given by
\[
\Phi_q(\upsilon_i)= \upsilon_{i+1}, \quad i=0, 1, \dots, n-1.
\]

Given a $q$-difference module $M$ cyclically generated by a vector $\upsilon_0$, we let $\upsilon_i=\Phi_q^i (\upsilon_0)$ for all $i\in\Z_+$. As $M$ is a finite dimensional $\C((x))$-vector space, there exists a smallest $n$ such that $\upsilon_0, \upsilon_1, \dots, \upsilon_n$ are $\C((x))$-linearly dependent (hence $\{\upsilon_i\mid 0\le i\le n-1\}$ is a $\C((x))$-basis of $M$). Thus there exist $a_i\in\C((x))$ for $0\le i\le n$ with $a_n\ne 0$ (following minimality of $n$) such that
\begin{eqnarray}\label{eq:dependence}
\upsilon_n+a_1 \upsilon_{n-1}+a_2\upsilon_{n-2}+\dots + a_n \upsilon_0=0.
\end{eqnarray}
Now $M$ is isomorphic to $\SD_q/\cJ_P$ for $P=y^n + a_1(x)  y^{n-1} + a_2(x) y^{n-2} +\dots+ a_n(x)$ with $v_0\mapsto 1+\cJ_P$.
\end{proof}
\begin{remark}
The relation \eqref{eq:dependence} implies
\[
\frac{1}{a_n(q x)}\upsilon_{n-1}+\frac{a_1(q x)}{a_n(q x)} \upsilon_{n-2}+\dots + \frac{a_{n-1}(q x)}{a_n(q x)}\upsilon_0+  \Phi_q^{-1}(\upsilon_0)=0.
\]
\end{remark}

\begin{example} \label{ex:simple-quot}
Continue Example \ref{ex:max-ideal} (1), where $P= y^2 - \lambda x^{-2\ell-1}$ for any positive integer $\ell$ and $\lambda\in\C^*$. Since $\cJ_P=\SD_q P$ is a maximal left ideal,
\[
\SD_q/\cJ_P= \C((x)) + \C((x)) y +\cJ_P
\]
is a simple  $\SD_q$-module.
Let $K=x^{2\ell+1} y^2$. Then for any $k\in\Z$,
\[
\baln
K \left(x^k +\cJ_P\right)= \lambda q^{-2 k} \left( x^k+ \cJ_P\right), \quad
K \left(x^k y +\cJ_P\right)= \lambda q^{-2 k+ 2\ell+1} \left( x^k y+ \cJ_P\right).
\ealn
\]
These eigenvectors of $K$ span $\SD_q/\cJ_P$ if we allow Laurent series in $x$.
\end{example}

%
%
%\subsection{Another construction of simple admissible $\cA_q(2)$-modules}
%
%

Generalising Example \ref{ex:simple-quot}, we obtain another construction of the simple admissible $\cA_q(2)$-modules $L_{(a, b)}[\lambda]$ for all $(a, b)\in \cP_{red}$ and $[\lambda]\in\C^*/q^\Z$.
\begin{theorem}\label{thm:construct-d-diff}
Consider the chain of subalgebras $\cA_q(2)\subset\cL_q(2)\subset \SD_q$.  For any $(a, b)\in \cP_{red}$ and $\lambda\in\C^*$, let $Q=y^b - \lambda x^{-a}$  and let $\cJ_Q=\SD_q Q$ be the left ideal generated by $Q$. Denote
$
\wh{L}_Q=\SD_q/\cJ_Q$,  and $L_Q=\cL_q(2)/(\cL_q(2)\cap \cJ_Q).
$
Then
\begin{enumerate}[i)]
\item
$\wh{L}_Q:=\SD_q/\cJ_Q$ is a simple $\SD_q$-module;
\item
$L_Q$ is a simple $\cL_q(2)$-module; and
\item the restriction $L_Q|_{\cA_q(2)}$ is isomorphic to the simple $\cA_q(2)$-module $L_{(a, b)}[\lambda]$.
\end{enumerate}
\end{theorem}
\begin{proof}
Note that
$
\wh{L}_Q=\SD_q/\cJ_Q=\sum_{i=0}^{b-1} \C((x))y^i + \cJ_Q.
$
Let $K=x^a y^b$, then
\beq\label{eq:K-eigen-vectors}
K( x^i y^\alpha +\cJ_Q)=\lambda q^{-i b + \beta a} ( x^i y^\alpha +\cJ_Q), \quad i\in\Z, \ \beta=\{0, 1, \dots, b-1\}.
\eeq
Since $a, b$ are co-prime,  $-i b + \beta a= -i' b + \beta' a$ if an only if there is some integer $k$ such that $(i-i', \beta-\beta')=k(a, b)$. Note that $|\beta-\beta'|< b$, thus we must have $k=0$, that is, $(i, \beta)=(i', \beta')$.  Hence the vectors $x^i y^\alpha +\cJ_Q$ have distinct eigenvalues.

Recall from Lemma \ref{lem:step-ops} that for the given $(a, b)\in\cP_{red}$, there exist integers $u, v$ satisfying $0\le u\le b$ and $0\le v\le a$ such that $\wh{X}=x^{a-v}y^{b-u}$ and $\wh{Y}=x^u y^v$ act on $\wh{L}_Q$ as step $1$-operators for $K$ eigenvalues. Thus they create all eigenvectors of $K$ from $1+\cJ_Q$, and also move any pair of eigenvectors to each other. Therefore, $\wh{L}_Q$ is a simple $\SD_q$-module.

We regard $L_Q$ as a $\C$-subspace of $\wh{L}_Q$, which
contains all the eigenvectors of $K$ given in \eqref{eq:K-eigen-vectors}. Clearly $L_Q$ is stable under the action of $\cL_q(2)$. The arguments with step-$1$ operators above imply that $L_Q$ is a simple $\cL_q(2)$-module.

Since the elements $K$, $\wh{X}$ and $\wh{Y}$ all belong to the subalgebra $\cA_q(2)\subset\cL_q(2)$, the proof of simplicity of $L_Q$ as $\cL_q(2)$-module above also shows that $L_Q|_{\cA_q(2)}$ is simple 
$\cA_q(2)$-module. This is the same argument used in proof of Theorem \ref{thm:restr} showing the simplicity of restrictions of some twisted simple $\cL_q(2)$-modules.  Since the $K$-eigenvalue  of  $1+\cJ_Q$ is $\lambda$,  we have $L_Q|_{\cA_q(2)}\cong L_{(a, b)}[\lambda]$.
\end{proof}

Let us now describe $
\wh{L}_Q=\SD_q/\cJ_Q
$
as a $q$-difference module.
The following result is an immediate consequence of Lemma \ref{lem:cyclic}.

\begin{corollary} Retain the setting of Theorem \ref{thm:construct-d-diff}.
Introduce the following elements $v_i = y^i +\cJ_Q$ (for $i=0, 1, \dots, b-1$) in
$
\wh{L}_Q=\SD_q/\cJ_Q.
$
Then the $\C$-linear map $\Phi_q: \wh{L}_Q \lra \wh{L}_Q$ defined by
\beq
&\Phi_q( f(x) v)= f(q^{-1} x) \Phi_q(v), \quad \forall v\in \wh{L}_Q, \ f\in\C((x)), \\
&\Phi_q(v_i)=v_{i+1}, \quad \Phi_q^{-1}(v_i)=v_{i-1}, \quad i=0, 1, \dots, b-1,   \\
&  \text{with} \quad  v_b = \lambda x^{-a} v_0, \quad  v_{-1} = q^{a} \lambda^{-1} x^a v_{b-1}, \nonumber
\eeq
yields a $q$-difference module structure for $\wh{L}_Q$, which is cyclically generated by $v_0$.
\end{corollary}
\begin{proof}
Only the expression for $v_{-1}$ requires explanation. Observe that
\[
\baln
&y^{-b} +\cJ_Q= q^{a b} \lambda^{-1} x^a +\cJ_Q, 
\ealn
\]
which can be obtained from the identity $y^{-b} (y^b -  \lambda x^{-a})  + \lambda q^{-a b} x^{-a} y^{-b}= 1$ by mapping both sides to the quotient.  Thus 
\[
y^{-1} +\cJ_Q= y^{b-1} y^{-b} +\cJ_Q = q^{a} \lambda^{-1} x^a y^{b-1}  +\cJ_Q.
\]
Hence
$
v_{-1}=
y^{-1} v_0 = q^{a} \lambda^{-1} x^a v_{b-1}.
$
%
%This proves the lemma.
\end{proof}

%%%%%%%%%%%%%%%%%%%%%%%%%%%%

\iffalse
\begin{example} Let
$\cJ_1=\SD_q P_1$, $\cJ_2=\SD_q P_2$ and $\cJ=\SD_q P$ with
\[
P_1=y -x, \quad P_2=y-1, \quad P=y - \frac{1- q^{-1} x^2}{1-x}y + \frac{x- q^{-1} x^2}{1-x}.
\]
In view of Remark \ref{rmk:factor}, we have
 $
%\[
\frac{\SD_q}\cJ=\frac{\cJ_1}\cJ \oplus \frac{\cJ_2}\cJ.
%\]
$
\end{example}
\fi

%%%%%%%%%%%%%%%%%%%%%%%%

%
%
\subsection{Relation to $q$-difference equations}

Consider $\mathbb{I}:=\SD_q/\SD_q(y-1)=\C((x))$, which is a left $\SD_q$-module  with the  action defined by
\beq\label{eq:module-I}
x.f(x) = x f(x), \quad y^{\pm 1}.f(x)= f(q^{\mp1} x), \quad \forall f\in \C((x)).
\eeq
For any $q$-difference module $(M, \Phi_q)$, we consider the $\C$-vector space $\Hom_{\SD_q}(\mathbb{I}, M)$. Each homomorphism $\eta$ is uniquely determined by $v_\eta:=\eta(1)\in M$.  Applying $\eta$ to $y.1=1$, we obtain
$
\eta(1) = \eta(y.1) = y.\eta(1) = \Phi_q(\eta(1)).
$
Hence $v_\eta$ obeys the $q$-difference equation
\beq
y. v_\eta = v_\eta,
\eeq
that is, $\left(\Phi_q -\id_M\right)(\eta(1))=0$.
 Also, every element  $v \in \ker\left(\Phi_q -\id_M\right)$ leads to a homomorphism $\eta$ such that $\eta(1)=v$. Therefore, as $\C$-vector space
\[
\Hom_{\SD_q}(\mathbb{I}, M) \cong \ker\left(\Phi_q -\id_M\right).
\]
This is usually called the solution of the $q$-difference module $(M, \Phi_q)$.

We can also describe $q$-difference modules in terms of $q$-differential equations.
For the monic polynomial $P$ of degree $n>0$ in $y$ over $\C((x))$ with non-zero ``constant'' term given in \eqref{eq:diff-op},  let $\cJ_P=\SD_q P$ be the left ideal of $\SD_q$ generated by $P$, and denote
$M_P:=\SD_q/\cJ_P$.
%To be more explicit, we still write
%\[
%P=y^n + a_1 y^{n-1} + a_2 y^{n-2} +\dots+ a_n, \quad  a_n\ne 0.
%\]
Then $v_i=y^i + \cJ_P$ with $i=0, 1, \dots, n-1$ form a basis of $M_P$ over $\C((x))$.  Any $v\in M_P$ can be expressed as $v=\sum_{i=0}^{n-1} z_i v_i$ with $z_i\in\C((x))$. The $q$-difference module structure map $\Phi_q$ is defined by
\beq
&\Phi_q(v) = \sum_i z_i(q^{-1} x) \Phi_q(v_i), 
\eeq
where
$\Phi_q(v_i)=v_{i+1}$ for $0\le i\le n-1$ with
$v_n = y^n + \cJ_P = -\sum_{i=0}^{n-1}  a_{n-i} v_i$.
Therefore,
\[
\Phi_q(v)  = \sum_{i=0}^{n-1} \left(z_{i-1} (q^{-1} x)  - a_{n-i}(x) z_{n-1}(q^{-1} x)\right)v_i, \quad z_{-1}=0.
\]

Let $j: M_P\stackrel{\sim}\lra \C((x))^n$ be the $\C((x))$-linear isomorphism defined by the above basis, thus $j(v)^T=Z^T=\bmtx z_0 & z_1 & \dots & z_{n-1} \emtx$. The  $q$-different module structure on $ \C((x))^n$ is now given by
\beq
\Psi_q (Z(x)) = A(x) Z(q^{-1} x) \label{eq:Phi-colum}
\eeq
with $\Psi_q=j \circ \Phi_q\circ j^{-1}$ and $A(x)= \bmtx  0 & 0 & 0 &\dots & 0 & -a_n(x)\\
						1 & 0 & 0 & \dots & 0 & -a_{n-1}(x)\\
						0 & 1 & 0 & \dots & 0 &- a_{n-2}(x)\\
						\vdots & \vdots & \vdots & \cdots &\vdots &\vdots \\
						0 & 0 & 0 & \dots & 1& -a_1(x)
			\emtx.
$
We also have the natural component wise $\SD_q$-action on $\C((x))^n=\underbrace{\mathbb{I}\oplus \mathbb{I}\oplus \dots\oplus \mathbb{I}}_n$ as described in Example \ref{eg:free-module}.
Then equation \eqref{eq:Phi-colum} becomes the following $q$-difference equation
\beq
y.Z =A^{-1} \Psi_q(Z).
\eeq

\medskip

%
%
%
%\subsection{Further comments on $q$-difference equations}
%
%
%
%\medskip
%\noindent
%{\bf Solutions of $q$-difference equations - examples}.
%
\noindent{\bf Some side remarks}.
We deviate from the main theme in this final bit of the paper to make some side remarks on  
$q$-difference equations of the form $Q f=0$, where $Q\in\SD_q$ and $f$ in $\C((x))$.
The study of such equations and their matrix generalisations in an analytic setting has a long history (see, for example, works of Jackson over 100 years ago \cite{J}). There was an explosion of research on the subject  related to quantum groups (see, e.g., \cite{BMZG, M}). An algebraic theory close to what we are interested in here can be found in, e.g.,  \cite{dV, P1, P2, RSZ}.

Given any $Q\in \SD_q$,  there is a unit $u$ (resp. $\tilde{u}$) such that $Q= u P$ (resp. $Q=\tilde{u}\wt{P}$), where $P$ (resp. $\wt{P}$)  is a monic polynomial in $y$ (resp. $y^{-1}$) with coefficients in $\C((x))$ with non-zero ``constant term''. Thus $P f=0$, and also $\wt{P} f=0$. Hence we only need to consider the $q$-difference operators which are polynomials in $y$ or $y^{-1}$ (but not both) over $\C((x))$, e.g., of the form
\eqref{eq:diff-op}. 

Let us present some easy examples of solutions of $q$-difference equations.

\begin{example}\label{lem:sol}
%\begin{lemma}\label{lem:sol}
For a given $a(x)\in\C((x))$ and an integer $n\ne 0$,  the equation
\beq\label{eq:de}
(y^n - a(x))f(x)=0
\eeq
has non-zero solutions in $\C((x))$ if and only if  $a(0)=q^{-n k}$ for some $k\in\Z$.
In this case, the non-zero solutions are given by
\[
f(x)=c x^{k} \exp(g(x)), \ \text{ with } \ g(x) = \sum_{i>0}\frac{\alpha_i x^i}{q^{-n i} -1}, \  c\in\C^*,
\]
where $\alpha_i\in\C$ are given by $\sum_{i>0} \alpha_i x^i=\ln\left(q^{nk} a(x)\right)$.
%\end{lemma}
\begin{proof}
For any non-zero $f(x)\in\C((x))$,
there exists some $k\in\Z$ such that $f(x)=x^k\sum_{i=0}^\infty f_i x^i$, where $f_i\in\C$ with $f_0\ne 0$.
If $f(x)$ is a non-zero solution of \eqref{eq:de}, then
\[
a(x) = f(q^{-n} x)/f(x), \quad a(0)=q^{-n k}.
\]
In this case,
$\alpha(x)=\ln\left(q^{nk} a(x)\right)$ is a well defined element of $x\C[[x]]$, which can be written as $\alpha(x)=\sum_{i>0} \alpha_i x^i$ with $\alpha_i\in\C$.  Define
$
g(x) = \sum_{i>0}\frac{\alpha_i x^i}{q^{-n i} -1}.
$
Then $g(q^{-n} x) = \alpha(x) + g(x)$. Hence $f(x)=c x^k \exp(g(x))$ for any $c\in\C^*$ satisfies
\[
f(q^{-n} x) = q^{-k n} \exp(\alpha(x)) f(x) = a(x) f(x),
\]
thus solves the $q$-difference equation \eqref{eq:de}.
\end{proof}
\end{example}

Below are two concrete cases of equation \eqref{eq:de}. 

\begin{example}\label{ex:de-sols}
\noindent
(1).
A simple case of equation \eqref{eq:de} is for $P=y^{-1} - (1+(q-1)x)$, which has the solution
$f(x)=c \exp_q(x)$ ($c\in\C^*$) with
$
\exp_q(x)=\sum_{j=0}^\infty \frac{x^j}{[j]_q!},
$
where $[j]_q!=[j]_q[j-1]_q\dots[1]_q$ with $[0]_q!=1$, and $[j]_q= \frac{q^{j}-1}{q-1}$ are the quantum integers defined in Section \ref{sect:q-deriv}.    Note that $\exp_q(x)$ is the celebrated $q$-exponential function.\\
%\end{example}

\noindent
(2).
%\begin{example}%[A $q$-difference Schr$\ddot{\rm o}$dinger equation] \label{ex:harmonic}
Let $H=\frac{1}{2}y^2 + V(x)$ with $V(x)\in\C((x))$, and consider the following analogue of the Schr$\ddot{\rm o}$dinger equation,
\beq\label{eq:schrodinger}
H \Psi(x) = E \Psi(x),
\eeq
which is the eigen-equation for $H$ with eigenvalues in $\C$ and eigenfunctions in $\C((x))$.  Example \ref{lem:sol} enables us to solve the problem completely.

For example, if $V(x)=\frac{1}{2} x^2$, we obtain the following eigenvalues  and the  corresponding eigenfunctions for $H$, 
\[
\baln
E_k= \frac{q^{-2k}}2, \quad \Psi_k(x) = c_k x^k\exp(g_k(x)),  \quad
g_k(x)=\sum_{i=1}^\infty \frac{(q^k x)^{2j}}{j(1-q^{-2j})},
\ealn
\]
where $k\in\Z$ and $c_k\in\C^*$.
\end{example}

Equation \eqref{eq:de} does not have any non-zero solution in $\C((x))$ if $a(0)\ne  q^{-n\ell}$ for any $\ell\in\Z$.
However, it may have non-zero solutions in $\C((x))[x^{\frac{1}{d}}, x^{-\frac{1}{d}}]$ for some $d|n!$. Here we need to make a choice of a $d$-th root of $q$, and extend the commutation relation between $x, y$ to
$
y x^{\pm \frac{1}{d}} = q^{\mp \frac{1}{d}} x^{\pm \frac{1}{d}}.
$

\begin{example}

%\begin{corollary}\label{ex:de-frac}
%Consider the following special case of the equation.
Fix co-prime integers $n, k$ both of which are non-zero, and let $a(x)$ be an element in $\C[[x]]$ such that $a(0)= q^{k}$. Then the equation
\begin{eqnarray}\label{eq:frac}
(y^n - a(x))f(x)=0
\end{eqnarray}
has non-trivial solutions in $\C((x))[x^{\frac{1}{n}}, x^{-\frac{1}{n}}]$ given by 
\[
f = c x^{-\frac{k}{n}}\exp\left(\sum_{i=1}^\infty \frac{\alpha_i x^i}{q^{-i n}-1}\right), \quad \text{with \ \ }
\sum_{i=1}^\infty \alpha_i x^i=\ln(q^{-k} a(x)).
\]
%\begin{proof}
This is easy to see. 
Let $f= x^{-\frac{k}{n}} g(x)$ with $g(x)\in\C((x))$, then equation \eqref{eq:frac} can be cast into
\[
(y^{n} - q^{-k} a(x))g(x)=0.
\]
Now $q^{-k} a(0)=1$, thus this $q$-difference equation has a unique (up to $\C^*$ multiples) non-zero solution, which is given by Example \ref{lem:sol}.
%\end{proof}
%\end{corollary}
\end{example}

\end{document}